\documentclass[11pt]{article}
\usepackage{mathrsfs}
\usepackage{amsmath}
\usepackage{bbm}
\usepackage{amsmath,amsthm,amssymb,amscd}
\usepackage{latexsym}
\usepackage{aliascnt}
\usepackage{hyperref}
\usepackage[numbers,sort&compress]{natbib}
\usepackage{CJK}
\usepackage{graphicx}
\usepackage[export]{adjustbox}
\usepackage{caption}
\newtheorem{Theorem}{Theorem}[section]
\newtheorem{Corollary}{Corollary}[section]
\newtheorem{Lemma}{Lemma}[section]
\newtheorem{Claim}{Claim}[section]
\newtheorem{Remark}{Remark}[section]
\theoremstyle{definition}
\newtheorem{Definition}{Definition}[section]
\newtheorem{Example}{Example}[section]

\numberwithin{equation}{section}

\begin{document}
\title{{\bf Asymptotic pseudo-periodicity of solutions for nonlinear fractional evolution equations}\thanks{The work was supported partly by the NSF of China (12371116, 12171094) and the Shanghai Key Laboratory for Contemporary Applied Mathematics (08DZ2271900).}}
\author{Jin Liang $^{a}$,  Yunyi Mu $^{b}$\thanks{Corresponding author, Email: muyy@sdju.edu.cn}, Gaston Mandata N'Gu\'{e}r\'{e}kata $^{c}$, Ti-Jun Xiao $^{d}$ \\
{\small $^a$ School of Mathematical Sciences, Shanghai Jiao Tong University}\\
{\small Shanghai 200240, China}\\
{\small $^b$ School of Integrated Circuits, Shanghai Dianji University}\\
{\small  Shanghai 201306, China}\\
{\small $^c$ School of Computer, Mathematical and Natural Sciences,}\\
{\small  Morgan State University, Baltimore, M.D. 21251, USA}\\
{\small $^d$ Shanghai Key Laboratory for Contemporary Applied Mathematics}\\
{\small School of Mathematical Sciences, Fudan University}\\
{\small Shanghai 200433, China}}

\date{}
\maketitle

\begin{abstract}
This work focuses on the dynamic behaviors, particularly the asymptotic periodicity, of solutions to nonlinear fractional-order dynamical systems modeled by Caputo fractional evolution equations in Banach spaces. We propose a new class of weighted pseudo $S$-asymptotically $(\xi,\tau,a)$-periodic functions, which unifies well-established function types such as standard periodic functions and $S$-asymptotically periodic functions as its special instances. We then systematically derive the fundamental properties of this new function family, and rigorously prove its completeness, corresponding convolution theorems, and nonlinear composition rules. Building on these newly established results, we demonstrate that under reasonable hypotheses, the mild solutions to a broad class of nonlinear fractional-order dynamical systems --- both delay-free and with finite delay --- in Banach spaces admit the (weighted pseudo) $S$-asymptotically $(\xi,\tau,a)$-periodic property, which further implies that no bifurcation or chaotic behavior emerges in these fractional-order systems. The theoretical conclusions are finally validated by illustrative examples describing memory-dependent complex dynamic processes, supplemented with supporting numerical simulations. Additionally, the presented weighted pseudo $S$-asymptotically $(\xi,\tau,a)$-periodic framework can be extended to functions with first-kind discrete discontinuities, so that systems with discontinuous nonlinear terms can be handled via the same analytical paradigm.

\vspace{0.5cm}


\vspace{0.4cm}

 \end{abstract}

\section{Introduction}

Fractional differential systems(FDSs), as a pivotal application of fractional calculus, have emerged as a powerful mathematical framework offering
more precise modeling capabilities than integer-order systems across diverse domains. These systems find extensive applications in complex material
characterization (fractal media, viscoelastic systems), emerging technologies (semiconductor physics, robotics control), fundamental sciences
(quantum mechanics, statistical physics), interdisciplinary fields (biomathematics, financial modeling), etc.
Recent years have witnessed substantial progress in utilizing FDSs to model real-world phenomena with memory effects.
For comprehensive theoretical foundations and applications on FDSs, readers may refer to
\cite{Abbas,Ansari,Biligiri,Chali5,Chali6,Chali8,Chali7,Borai,Hilfer,Hilferl,Kilbas,Kulkarni,LiL1,Liangm5,Liu,Luo,Mainardi,Pagnini,Podlubny}
 and the references therein.

It is well-known that the qualitative analysis of long-term  dynamic behaviors, such as asymptotic periodicity, stability, and bifurcation, is central to the study of nonlinear dynamical systems. Many physical and engineered systems governed by fractional-order models exhibit asymptotic or quasi-periodic oscillations rather than strict periodicity, due to complex interplays of memory, damping, and external perturbations. Investigating these generalized periodic patterns is crucial for understanding system robustness, performance, and control strategies. While considerable attention has been paid to classical periodic, almost periodic, and pseudo-almost periodic solutions in integer-order dynamics, the corresponding theoryfor fractional-order nonlinear evolution equations in infinite-dimensional settings (Banach spaces) remains less explored, particularly concerning weighted and $\emph{S}$-asymptotically type periodicities.
Key generalizations of periodic functions include almost periodic functions, pseudo almost periodic functions, and weighted pseudo almost periodic functions,
 $\emph{S}$-asymptotically $\tau$-periodic functions, etc.
Research on these generalizations of periodic functions can be found in references \cite{Blot,Bochner,Henri}.
 In \cite{Pierri}, the authors proposed the notion of pseudo $\emph{S}$-asymptotically periodic functions, investigated its qualitative aspects,
 and explored the existence of pseudo $\emph{S}$-asymptotically periodic mild solutions for a class of first-order abstract neutral differential equations.
  In \cite{LiL1}, the authors investigated the existence and uniqueness of $\emph{S}$-asymptotically $\tau$-periodic solutions to fractional differential
  equations of order $q\in (0,1)$ with finite delay in a Banach space $\mathbb{E}$. They also applied their existence and uniqueness results to derive
  the $\emph{S}$-asymptotically $\tau$-periodic solutions for fractional-order autonomous neural networks with delay. In \cite{Alvarez},
  the authors analyzed the $(\tau,a)$-periodic functions by examining the property $\nu(\lambda+\tau)=a\nu(\lambda)\ (a\in \mathbb{C})$ of any solution $\nu(\lambda)$ of
  the well-known Mathieu's equations $\nu^{\prime\prime}+a\nu=2p{\rm{cos}}(2t)\nu$, which is a Hill's equation with a single harmonic mode.
  It is evident that when $a= 1$, the $(\tau,a)$-periodic functions reduce to the classic $\tau$-periodic functions, and when $a=-1$,
  they reduce to the $\tau$-antiperiodic functions. For more interesting works for generalizations of periodic functions and their applications
  in various differential equations, readers may refer to \cite{Agar2,Cieutat,Cue,Cue0,dos1,Oueama} and the references therein.

Stimulated by the works above, in this paper, we will firstly investigate the $w$-pseudo $\emph{S}$-asymptotically $(\xi,\tau,a)$-periodic solutions
for the following nonlinear Caputo fractional evolution system in a Banach space $\mathbb{E}$:
\begin{equation}\label{equation2}
\left\{
\begin{array}{lr}
\left(^c\mathcal{D}^{\tilde{\iota}}_\lambda \nu\right)(\lambda)=\Upsilon\nu(\lambda)+\varrho(\lambda,\nu(\lambda)),\ \lambda>0,\\
\nu(0)=\nu_0,
\end{array}
\right.
\end{equation}
where $^c\mathcal{D}^{\tilde{\iota}}_\lambda\ (0<\tilde{\iota}<1)$ is the Caputo fractional derivative of order $\tilde{\iota}$;  the closed operator
 $\Upsilon$ ($D(\Upsilon)\subseteq \mathbb{E}$) is the infinitesimal generator of a $C_0$ semigroup $\Psi(\lambda)$ on $[0,\infty)$;
 $\nu_0\in \mathbb{E}$; $\varrho:\ [0,\infty)\times \mathbb{E}\rightarrow \mathbb{E}$ is a given function to be specified  later.

Secondly, we will consider the $\emph{S}$-asymptotically $(\tau,a)$-periodic solutions for the nonlinear Caputo fractional evolution
system with finite delay below in a Banach space $\mathbb{E}$:
\begin{equation}\label{equation1}
\left\{
\begin{array}{lr}
\left(^c\mathcal{D}^{\tilde{\iota}}_\lambda \nu\right)(\lambda)=\Upsilon\nu(\lambda)+\varrho(\lambda,\nu_\lambda),\ \lambda>0,\\
\nu(\lambda)=\eta(\lambda),\qquad \qquad \qquad \qquad \ \lambda\in [-\kappa,0],
\end{array}
\right.
\end{equation}
where $^c\mathcal{D}^{\tilde{\iota}}_\lambda\ (0<\tilde{\iota}<1)$  is the Caputo fractional derivative of order $\tilde{\iota}$;  $\kappa>0$;
the closed operator $\Upsilon$ ($D(\Upsilon)\subseteq \mathbb{E}$) generates a $C_0$ semigroup $\Psi(\lambda)$ on $[0,\infty)$;
 $\varrho:\ [0,\infty)\times C([-\mu,0],\mathbb{E})\rightarrow \mathbb{E}$ is a given function to be defined subsequently;
 $\nu_\lambda:\ [-\kappa,0]\rightarrow \mathbb{E}$ is defined by
$\nu_\lambda(\theta)=\nu(\lambda+\theta)$ for $\theta\in [-\kappa,0]$, and $\eta\in C([-\kappa,0],\mathbb{E})$.

The research on $\emph{S}$-asymptotically $(\tau,a)$-periodic solutions for fractional evolution system in Banach spaces is blank,
 and this paper fills this gap. The rest of this paper is organized as follows. In Section 2, we give some notations and recall some
  basic known results which will be used throughout the paper. In Section 3, we introduce the notion of weighted pseudo $\emph{S}$-asymptotically
  $(\xi,\tau,a)$-periodic functions and present some basic properties of such functions. In Section 4, we investigate the weighted pseudo
  $\emph{S}$-asymptotically $(\xi,\tau,a)$-periodic mild solutions for time-fractional Caputo evolution nonlinear differential system
  $(\ref{equation2})$ on $[0,\infty)$ in a Banach space $\mathbb{E}$. In Section 5, we discuss the $\emph{S}$-asymptotically $(\tau,a)$-periodic
   mild solutions for Caputo fractional nonlinear evolution differential system $(\ref{equation1})$ on $[0,\infty)$ with finite delay in a
   Banach space $\mathbb{E}$. Finally, in Section 6, two examples are presented to show the applicability of the abstract results.

\section{Preliminaries}

 We begin by establishing basic notation for subsequent analysis. Let $\mathbb{C}$ , $\mathbb{R}$ and $\mathbb{N}_+$ be the sets of complex numbers, real numbers and positive integers. Consider two complex Banach spaces $(\mathbb{E}, \|\cdot\|_{\mathbb{E}})$ and $(\mathbb{F}, \|\cdot\|_{\mathbb{F}})$.
  When no confusion arises, we denote the norm $\|\cdot\|_{\mathbb{E}}$ simply by $\|\cdot\|$.
 We define $C(\mathbb{R}, \mathbb{E})$ and $C([0,\infty), \mathbb{E})$ to be the spaces of all continuous functions mapping $\mathbb{R}$ and $[0,\infty)$ to $\mathbb{E}$, respectively.
We denote $ C^1([0,\infty),\mathbb{E})$ the space of all continuously differentiable functions from $[0,\infty)$ to $\mathbb{E}$.
 We also write $C(\mathbb{R} \times \mathbb{F}, \mathbb{E})$ for the family of all continuous functions $h: \mathbb{R} \times \mathbb{F} \to \mathbb{E}$.
 The space of all bounded linear transformations from $\mathbb{E}$ to itself is represented by $\mathcal{B}(\mathbb{E})$.
 The collection of all bounded continuous mappings from $\mathbb{R}$ to $\mathbb{E}$ is denoted by $BC(\mathbb{R}, \mathbb{E})$, equipped with the supremum norm $\|h\|_{\infty} = \mathop{\sup}\limits_{\theta \in \mathbb{R}} \|h(\theta)\|_{\mathbb{E}}$.
The space of bounded continuous functions from the non-negative real line $[0, \infty)$ to $\mathbb{E}$ is $C_b([0, \infty), \mathbb{E})$, equipped with the norm $\|h\|_{\infty} = \mathop{\sup}\limits_{\theta \geq 0} \|h(\theta)\|_{\mathbb{E}}$.
For a fixed $\kappa > 0$, $C([-\kappa, 0], \mathbb{E})$ represents the Banach space of continuous functions from the interval $[-\kappa, 0]$ to $\mathbb{E}$, with the norm $\|h\|_{[-\kappa,0]} =  \mathop{\sup}\limits_{\theta \in [-\kappa,0]} \|h(\theta)\|_{\mathbb{E}}$.
Given a function $h \in C([-\kappa, \infty), \mathbb{E})$ and a point $\lambda \geq 0$, we define the history segment $h_\lambda \in C([-\kappa, 0], \mathbb{E})$ by $h_\lambda(\theta) := h(\lambda + \theta)$ for all $\theta \in [-\kappa, 0]$.

We now present essential definitions and properties concerning weight functions and $(\tau, a)$-periodic functions. For more details, the reader is referred to \cite{Alvarez,Liangm7}.

For a function $h: \mathbb{R} \to \mathbb{E}$, we define its translation operator $h_r(\lambda) := h(\lambda + r)$ for each $r \in \mathbb{R}$.
Let $\mathcal{W}$ represent the collection of all locally integrable weight functions $w: \mathbb{R} \to [0, \infty)$ that are positive almost everywhere.
 We denote $\mathcal{W}_{b}$ as the set of all functions $w$ belonging to $\mathcal{W}$ which are bounded and possess a positive infimum over $\mathbb{R}$, i.e., $\displaystyle{\inf_{\theta \in \mathbb{R}}} w(\theta) > 0$.
 For $w \in \mathcal{W}$ and $\ell > 0$, we define $$\sigma(\ell, w) := \displaystyle{\int_{-\ell}^{\ell} w(\theta) d\theta}, \quad \mbox{or} \quad \displaystyle{\int_{0}^{\ell} w(\theta) d\theta}$$ according to the context.
 In this study, we assume that any weight function $w$ involved belongs to $\mathcal{W}$.

We introduce the following functional spaces for later use:
$$\mathcal{B}_w := \left\{h: \mathbb{R} \rightarrow [0,\infty) \Big | \ h \text{ is locally integrable and } \lim_{\ell \to \infty} \frac{1}{\sigma(\ell, w)} \int_{-\ell}^{\ell} h(\theta) w(\theta) d\theta < \infty \right\},$$
and
$$L^q(\mathbb{R}, [0,\infty), w) := \left\{h: \mathbb{R} \to [0,\infty) \Big | \ h \text{ is locally integrable and } \int_{-\infty}^{\infty} h^q(\theta) w(\theta) d\theta < \infty \right\},$$
where $q \in [1, \infty)$.

\begin{Definition}
For $w_i \in \mathcal{W}$ ($i = 1, 2$), we say that $w_1$ is equivalent to $w_2$, denoted $w_1 \sim w_2$, if and only if their quotient $w_1 / w_2$ belongs to $\mathcal{W}_b$.
\end{Definition}

\begin{Definition}
For any given $0\neq a \in \mathbb{C} $ and $\tau >0$, a function $h \in C(\mathbb{R}, \mathbb{E})$ is said to be $(\tau, a)$-periodic if it satisfies the condition
$$h(\lambda + \tau) = a h(\lambda), \quad  \lambda \in \mathbb{R}.$$
In this case, $\tau$ is called the $a$-period of $h$.
\end{Definition}

We denote by $\mathcal{P}_{\tau,a}(\mathbb{R}, \mathbb{E})$ the set of all $(\tau,a)$-periodic mappings from $\mathbb{R}$ to $\mathbb{E}$.
In the particular cases where $a=1$ and $a=-1$, this set reduces to the space of $\tau$-periodic and $\tau$-antiperiodic functions from $\mathbb{R}$ to $\mathbb{E}$, respectively.
For conciseness, the notation $\mathcal{P}_\tau(\mathbb{R}, \mathbb{E})$ is frequently used instead of $\mathcal{P}_{\tau,1}(\mathbb{R}, \mathbb{E})$.

Utilizing the principal logarithm, we define $$a^ \uparrow (\lambda):=a^{\frac{\lambda}{\tau}} = \exp\left(\frac{\lambda}{\tau}Log(a)\right).$$
Additionally, we employ the notation $$|a|^\uparrow(\lambda) = |a^\uparrow(\lambda)| = |a|^{\frac{\lambda}{\tau}}.$$

\begin{Lemma}\label{dbhjddkd}
Let $h \in C(\mathbb{R}, \mathbb{E})$. Then $h \in \mathcal{P}_{\tau,a}(\mathbb{R}, \mathbb{E})$ if and only if $h(\lambda) = a^\uparrow(\lambda) h_1(\lambda)$, where $h_1(\lambda) \in \mathcal{P}_\tau(\mathbb{R}, \mathbb{E})$.
\end{Lemma}

\begin{Lemma}\label{dnfjk}
The set $\mathcal{P}_{\tau,a}(\mathbb{R}, \mathbb{E})$ forms a Banach space under the norm
$$\|h\|_{\tau a} = \sup_{\lambda \in [0,\tau]} \|a^\uparrow(-\lambda) h(\lambda)\|.$$
\end{Lemma}

The definitions of mild solutions for Eqs. $(\ref{equation2})$ and $(\ref{equation1})$ can be found in references such as \cite{Borai} and \cite{LiL1}.
\begin{Definition}
For a function $h\in C([0,\infty),\mathbb{E})$, its fractional integral of order $\tilde{\iota}\ (0<\tilde{\iota}<1)$ with initial point zero is expressed as
$$\mathcal{I}_\lambda^{\tilde{\iota}} h(\lambda)=\frac{1}{\Gamma(\tilde{\iota})}\int_0^\lambda (\lambda-\theta)^{\tilde{\iota}-1}h(\theta)d\theta,\ \lambda>0,$$
where $\Gamma(\cdot)$ denotes the gamma function.
\end{Definition}

\begin{Definition}
For a function $h\in C^1([0,\infty),\mathbb{E})$, its Caputo fractional derivative of order $\tilde{\iota}\ (0<\tilde{\iota}<1)$ with initial point zero can be formulated as
$$^c\mathcal{D}_\lambda^{\tilde{\iota}} h(\lambda)=\frac{1}{\Gamma(1-\tilde{\iota})}\int_0^\lambda \frac{h^\prime(\theta)}{(\lambda-\theta)^{\tilde{\iota}}}d\theta,\ \lambda>0.$$
\end{Definition}

\begin{Definition}
A function $\nu\in C([0,\infty),\mathbb{E})$ satisfying the integral equation
\begin{equation}\label{dnkfhfkhis1}
\nu(\lambda)=\displaystyle{\Psi_1(\lambda)\nu_0+\int_0^\lambda(\lambda-\theta)^{\tilde{\iota}-1}\Psi_2(\lambda-\theta)\varrho(\theta,\nu(\theta))d\theta,\ \lambda\geq0,}
\end{equation}
is called a mild solution of Eq. $(\ref{equation2})$, where
$$\Psi_1(\lambda)=\int_0^\infty \vartheta_{\tilde{\iota}}(\theta)\Psi(\lambda^{\tilde{\iota}} \theta)d\theta,\ \Psi_2(\lambda)=\tilde{\iota}\int_0^\infty \theta\vartheta_{\tilde{\iota}}(\theta)\Psi(\lambda^{\tilde{\iota}} \theta)d\theta.$$
Here, $\vartheta_{\tilde{\iota}}(\cdot)$ is a probability density function on $(0,\infty)$ defined by
$$\vartheta_{\tilde{\iota}}(\lambda)=\frac{1}{\pi}\displaystyle{\sum_{n=1}^\infty} (-\lambda)^{n-1}\frac{\Gamma(n\tilde{\iota})}{(n-1)!}{\rm{sin}}(n\pi \tilde{\iota})\geq 0,\ \lambda\in (0,\infty).$$
\end{Definition}
\begin{Remark}
If the semigroup $\{\Psi(\lambda)\}_{\lambda\geq 0}$ is uniformly bounded, that is, there exists an $M>0$ such that $\|\Psi(\lambda)\|\leq M,\ \lambda\geq 0$, then
\begin{equation}\label{dnkhfkhis3}
\|\Psi_1(\lambda)\|\leq M ,\ \|\Psi_2(\lambda)\|\leq \frac{M}{\Gamma(\tilde{\iota})},\ \lambda\geq 0.
\end{equation}
Moreover, if $\{\Psi(\lambda)\}_{\lambda\geq 0}$ is a uniformly exponential stable analytic semigroup, i.e., there exist positive constants $M$ and $\alpha$ such that
$$\|\Psi(\lambda)\|\leq M{\rm{e}}^{-\alpha \lambda},\ \lambda\geq 0,$$
then we have
\begin{equation}\label{dnkfhfkhis2}
\lim_{\lambda\rightarrow \infty}\|\Psi_1(\lambda)\|=0,
\end{equation}
and
\begin{eqnarray}
\int_0^\lambda(\lambda-\theta)^{\tilde{\iota}-1}\|\Psi_2(\lambda-\theta)\|d\theta&\leq& \tilde{\iota} M\int_0^\lambda\int_0^\infty \theta_1\vartheta_{\tilde{\iota}}(\theta_1)(\lambda-\theta_2)^{\tilde{\iota}-1}{\rm{e}}^{-\alpha(\lambda-\theta_2)^{\tilde{\iota}} \theta_1}d\theta_1 d\theta_2\nonumber\\
&\leq&M\int_0^\infty \vartheta_{\tilde{\iota}}(\lambda)d\lambda \int_0^\infty {\rm{e}}^{-\alpha \lambda}d\lambda=\frac{M}{\alpha}.\label{hejek}
\end{eqnarray}
In this case, we say the semigroup $\{\Psi(\lambda)\}_{t\geq 0}$ is of type $\mathbf{T}(M,\alpha)$.
\end{Remark}
\begin{Definition}
A function $\nu\in C([-\kappa,\infty),\mathbb{E})$ that fulfills the equation
\begin{equation}\label{equation12}
\nu(\lambda)=
\left\{
\begin{array}{lr}
\eta(\lambda),\ \qquad \qquad \qquad \qquad \qquad \qquad \qquad \qquad  \ \ \ \ \lambda\in [-\kappa,0],\\
\displaystyle{\Psi_1(\lambda)\eta(0)+\int_0^\lambda(\lambda-\theta)^{\tilde{\iota}-1}\Psi_2(\lambda-\theta)\varrho(\theta,\nu_\theta)d\theta,\ \lambda>0,}
\end{array}
\right.
\end{equation}
is referred to as a mild solution of Eq. $(\ref{equation1})$.
\end{Definition}

\section{Weighted pseudo $\emph{S}$-asymptotically $(\xi,\tau,a)$-periodic functions and their properties}

\subsection{$\emph{S}$-asymptotically $(\xi,\tau,a)$-periodic functions}

Throughout the paper, we assume the function $\xi(\lambda)\in C(\mathbb{R},\mathbb{C})$ with $\xi(\lambda)\neq 0$.
\begin{Lemma}\label{hdkjd}
Assume the function $\xi(\lambda)$ satisfies the condition $\displaystyle{\sup_{\lambda\in \mathbb{R}}} \left|\xi(\lambda)a^\uparrow(\lambda)\right|<\infty$.
Then, the linear space $\mathcal{P}_{\tau,a}(\mathbb{R},\mathbb{E})$ equipped with the norm
$$\|h\|_{b\xi}=\displaystyle{\sup_{\lambda\in \mathbb{R}}}\|\xi(\lambda)h(\lambda)\|$$
forms a Banach space.
\end{Lemma}

\begin{proof}
By invoking Lemma \ref{dbhjddkd} and the stated assumptions, it is straightforward to confirm that $\|\cdot\|_{b\xi}$ defines a norm
on the linear space $\mathcal{P}_{\tau,a}(\mathbb{R},\mathbb{E})$. To establish completeness, we select an arbitrary Cauchy sequence
 $\{h_n\}_{n=1}^\infty$ in $\mathcal{P}_{\tau,a}(\mathbb{R},\mathbb{E})$, where each element satisfies $$h_n(\lambda) = a^\uparrow(\lambda)h_n'(\lambda)$$
 with $h_n'(\lambda) \in \mathcal{P}_\tau(\mathbb{R},\mathbb{E})$ for $n \in \mathbb{N}_{+}$.

It is evident that the sequence $\{h_n'(\lambda)\}_{n=1}^\infty$ is a Cauchy sequence in the space $BC(\mathbb{R},\mathbb{E})$.
Consequently, there exists a function $h'(\lambda) \in \mathcal{P}_\tau(\mathbb{R},\mathbb{E})$ such that $\|h_n' - h'\|_\infty \to 0$ as $n \to \infty$.
Define $h(\lambda) := a^\uparrow(\lambda)h'(\lambda)$. It follows that $h \in \mathcal{P}_{\tau,a}(\mathbb{R},\mathbb{E})$ and $$\|h_n - h\|_{b\xi} \to 0 \quad \mbox{as} \quad n \to \infty,$$
thereby concluding the proof of this lemma.
\end{proof}

In the particular case where $\xi(\lambda) = a^\uparrow(-\lambda)$, Lemma \ref{hdkjd} reduces to Lemma \ref{dnfjk}. It is evident that the function $\xi(\lambda) = (a^\uparrow(\lambda) + \lambda^{2})^{-1}$ also fulfills the conditions of Lemma \ref{hdkjd}.

Based on Lemma \ref{hdkjd}, we introduce the following function space for any given function $\xi(\lambda)$:
$$
BC_\xi(\mathbb{R}, \mathbb{E}) = \left\{ h \in C(\mathbb{R}, \mathbb{E}) : \sup_{\lambda \in \mathbb{R}} \| \xi(\lambda) h(\lambda) \| < \infty \right\}.
$$

\begin{Lemma}\label{jjdjdj}
Let $h \in C(\mathbb{R}, \mathbb{E})$. Then $h \in BC_\xi(\mathbb{R}, \mathbb{E})$ if and only if there exists $h_1 \in BC(\mathbb{R}, \mathbb{E})$ such that $h(\lambda) = \frac{1}{\xi(\lambda)} h_1(\lambda)$ for all $\lambda \in \mathbb{R}$.
\end{Lemma}

\begin{proof}
(Necessity) Suppose $h \in BC_\xi(\mathbb{R}, \mathbb{E})$. Define $h_1(\lambda) = \xi(\lambda) h(\lambda)$. By definition of $BC_\xi(\mathbb{R}, \mathbb{E})$, $h_1 \in BC(\mathbb{R}, \mathbb{E})$, and clearly $h(\lambda) = \frac{1}{\xi(\lambda)} h_1(\lambda)$.

(Sufficiency) Assume $h(\lambda) = \frac{1}{\xi(\lambda)} h_1(\lambda)$ with $h_1 \in BC(\mathbb{R}, \mathbb{E})$. Since $\xi(\lambda) \neq 0$ for all $\lambda$, $h \in C(\mathbb{R}, \mathbb{E})$. Moreover,
$$
\sup_{\lambda \in \mathbb{R}} \| \xi(\lambda) h(\lambda) \| = \sup_{\lambda \in \mathbb{R}} \| h_1(\lambda) \| = \| h_1 \|_\infty < \infty,
$$
which implies $h \in BC_\xi(\mathbb{R}, \mathbb{E})$.
\end{proof}

\begin{Lemma}
The space $BC_\xi(\mathbb{R}, \mathbb{E})$ is a Banach space when equipped with the norm
$$
\| h \|_{b\xi} = \sup_{\lambda \in \mathbb{R}} \| \xi(\lambda) h(\lambda) \|.
$$
\end{Lemma}

\begin{proof}
For any given Cauchy sequence $\{h_n\}_{n=1}^\infty \subset BC_{\xi}(\mathbb{R},\mathbb{E})$, it suffices to demonstrate the existence of an element
 $h \in BC_{\xi}(\mathbb{R},\mathbb{E})$ satisfying $$\displaystyle{\lim_{n\rightarrow \infty}}\|h_n-h\|_{b\xi}=0.$$ According to Lemma \ref{jjdjdj},
 each element can be expressed as $$h_n(\lambda)=\frac{1}{\xi(\lambda)}h_{1n}(\lambda)$$ with $h_{1n}(\lambda) \in BC(\mathbb{R},\mathbb{E})$. Consequently, the sequence $\{h_{1n}\}_{n=1}^\infty$
 forms a Cauchy sequence within $BC(\mathbb{R},\mathbb{E})$. Therefore, there exists a function $h_1(\lambda) \in BC(\mathbb{R},\mathbb{E})$
 such that $\displaystyle{\lim_{n\rightarrow \infty}}\|h_{1n}-h_1\|_{\infty}=0$. Defining $h(\lambda):=\frac{1}{g(\lambda)}h_1(\lambda)$,
 we obtain $h(t) \in BC_{\xi}(\mathbb{R},\mathbb{E})$ and $\displaystyle{\lim_{n\rightarrow \infty}}\|h_n-h\|_{bg}=0$. This concludes the proof of the lemma.
\end{proof}
Next, we will construct (weighted pseudo) $\emph{S}$-asymptotically $(\xi,\tau,a)$-periodic functions in the Banach space $BC_{\xi}(\mathbb{R},\mathbb{E})$. The conditions on the function $\xi(\cdot)$ are quite weak, and no relation between $\xi$ and $\tau,\ a$ is required.

In the special case where $\xi(\lambda)=a^\uparrow(-\lambda)$, the normed space $(BC_{\xi},\|\cdot\|_{b\xi})$ is denoted by $(BC_{\tau,a},\|\cdot\|_{b\tau a})$.

\begin{Definition}
Let $h$ be a function such that both $h(\cdot)$ and $h(\cdot+\tau)$ belong to $BC_{\xi}(\mathbb{R},\mathbb{E})$. Such a function is termed \emph{S}-asymptotically $(\xi,\tau,a)$-periodic provided
$$
\lim_{\lambda\rightarrow \pm\infty}\left\|\xi(\lambda)[h(\lambda+\tau)-ah(\lambda)]\right\|=0.
$$
\end{Definition}

We denote the set of such functions by ${\rm{SAP}}_{\xi,\tau,a}(\mathbb{R},\mathbb{E})$. In the case where $\xi(t)=a^\uparrow(-\lambda)$, the notation ${\rm{SAP}}_{\xi,\tau,a}(\mathbb{R},\mathbb{E})$ is simplified to ${\rm{SAP}}_{\tau,a}(\mathbb{R},\mathbb{E})$.

For  $\xi_1(\lambda)=a^\uparrow(-\lambda)$ and $\xi_2(\lambda)=(a^\uparrow(\lambda)+\lambda^{2})^{-1}$, we have $${\rm{SAP}}_{\xi_1,\tau,a}(\mathbb{R},\mathbb{E})\neq {\rm{SAP}}_{\xi_2,\tau,a}(\mathbb{R},\mathbb{E}).$$ More precisely, as $\lambda\rightarrow +\infty$ or $\lambda\rightarrow -\infty$, the functions $\xi_1(\lambda)$ and $\xi_2(\lambda)$ exhibit equivalence in one limit direction but not in the other.

\begin{Lemma}
Let $h_i\ (i=1,2)$ and $h$ belong to ${\rm{SAP}}_{\xi,\tau,a}(\mathbb{R},\mathbb{E})$. The listed properties below are then satisfied:\\
(1) $h_1+h_2\in {\rm{SAP}}_{\xi,\tau,a}(\mathbb{R},\mathbb{E})$, and, for every $c\in \mathbb{C}$, $ch\in {\rm{SAP}}_{\xi,\tau,a}(\mathbb{R},\mathbb{E})$ .\\
(2) Suppose for each $r\in \mathbb{R}$, $|\xi_r|\sim |\xi|$. Then, for any $r\in \mathbb{R}$, $h_r(\lambda)\in {\rm{SAP}}_{\xi,\tau,a}(\mathbb{R},\mathbb{E})$.
\end{Lemma}

\begin{proof}
(1) is evident from the definition.\\
(2) Observe that
\begin{eqnarray*}
&&\displaystyle{\lim_{\lambda\rightarrow \pm\infty}}\left\|\xi(\lambda)[h(\lambda+r+\tau)-ah(\lambda+r)]\right\|\\
&=&\displaystyle{\lim_{\lambda\rightarrow \pm\infty}}\left\|\xi(\lambda+r)[h(\lambda+r+\tau)-ah(\lambda+r)]\right\|\cdot\left|\frac{\xi(\lambda)}{\xi(\lambda+r)}\right|.
\end{eqnarray*}
By applying the given conditions, the conclusion follows.
\end{proof}

\begin{Lemma}
Assume for every $ r \in \mathbb{R} $, the function $\xi$ satisfies $ |\xi_r| \sim |\xi| $.
Then the linear space $ \left( \mathrm{SAP}_{\xi,\tau,a}(\mathbb{R}, \mathbb{E}), \, \|\cdot\|_{b\xi} \right) $ is a Banach space.
\end{Lemma}
\begin{proof}
It suffices to prove the completeness of the linear space ${\rm{SAP}}_{\xi,\tau,a}(\mathbb{R},\mathbb{E})$ under the norm $\|\cdot\|_{b\xi}$. To this end,  let $ \{h_n\}_{n=1}^\infty$  be a Cauchy sequence in $ \mathrm{SAP}_{\xi,\tau,a}(\mathbb{R}, \mathbb{E}) $ with respect to the norm $ \|\cdot\|_{b\xi} $. Since $ \{h_n\}_{n=1}^\infty$ is also a Cauchy sequence in $BC_{\xi}(\mathbb{R},\mathbb{E})$, the sequence $ \{h_n\} $ converges uniformly in $ BC_\xi(\mathbb{R}, \mathbb{E}) $ to some $ h \in BC_\xi(\mathbb{R}, \mathbb{E}) $.  $|\xi_\tau|\sim |\xi|$, so there exists $C>0$ such that $$\left|\frac{\xi(\lambda)}{\xi(\lambda+\tau)}\right|\leq C,\qquad \lambda\in \mathbb{R}.$$
Let $ \mu > 0 $ be arbitrary. By the Cauchy property, there exists $ N \in \mathbb{N}_{+} $ such that for all $ n \geq N$,
$$
\|h_n - h\|_{b\xi} < \frac{\mu}{3} \min\left\{ \frac{1}{C}, \frac{1}{|a|} \right\}.
$$
By the definition of $ \mathrm{SAP}_{\xi,\tau,a}(\mathbb{R},\mathbb{E}) $, there exists $ T_1 > 0 $ such that for all $ |\lambda| > T_1 $,
$$
\|\xi(\lambda)[h_N(\lambda + \tau) - a \, h_N(\lambda)]\| < \frac{\mu}{3}.
$$
Now, for $ |\lambda| > T_1 $, we estimate:
\begin{eqnarray*}
\|\xi(\lambda)[h(\lambda+\tau)-ah(\lambda)]\|&\leq&\|\xi(\lambda)[h(\lambda+\tau)-h_N(\lambda+\tau)]\|+\|\xi(\lambda)[h_N(\lambda+\tau)-ah_N(\lambda)]\|\\
&&+\|\xi(\lambda)[ah_N(\lambda)-ah(\lambda)]\|\\
&=&\left|\frac{\xi(\lambda)}{\xi(\lambda+\tau)}\right|\cdot\|\xi(\lambda+\tau)[h(\lambda+\tau)-h_N(\lambda+\tau)]\|\\
&&+\|\xi(\lambda)[h_N(\lambda+\tau)-ah_N(\lambda)]\|+|a|\cdot\|\xi(\lambda)[h_N(\lambda)-h(\lambda)]\|\\
&\leq&C\cdot\frac{\mu}{3C}+\frac{\mu}{3}+|a|\cdot\frac{\mu}{3|a|}=\mu.
\end{eqnarray*}
Therefore, $ h \in \mathrm{SAP}_{\xi,\tau,a}(\mathbb{R}, \mathbb{E}) $, which proves the completeness of the space. We have thus finished the proof of this lemma.
\end{proof}

The functions $\xi(\lambda)=a^\uparrow(-\lambda)$ and $(a^\uparrow(\lambda)+\lambda^2)^{-1}$ both satisfy the condition $ |\xi_r| \sim |\xi| $ for all $ r \in \mathbb{R} $.

\subsection{Weighted pseudo $\emph{S}$-asymptotically $(\xi,\tau,a)$-periodic functions}

We introduce the following definition of weighted pseudo $\emph{S}$-asymptotically $(\xi,\tau,a)$-periodic functions.
\begin{Definition}\label{hddj1}
 For a function $h$ such that $h(\cdot),\ h(\cdot+\tau)\in BC_{\xi}(\mathbb{R},\mathbb{E})$, it is termed weighted pseudo $\emph{S}$-asymptotically $(\xi,\tau,a)$-periodic with respect to the weight function $w$, or $w$-pseudo $\emph{S}$-asymptotically $(\xi,\tau,a)$-periodic, provided the condition below is fulfilled
$$\lim_{\ell\rightarrow \infty}\frac{1}{\sigma(\ell,w)}\int_{-\ell}^\ell \left\|\xi(\lambda)[h(\lambda+\tau)-ah(\lambda)]\right\|w(\lambda)d\lambda=0,$$
where $\sigma(\ell,w)=\displaystyle{\int_{-\ell}^\ell w(\theta)d\theta}.$
\end{Definition}
${\rm{PSAP}}_{\xi,\tau,a}(\mathbb{R},\mathbb{E},w)$ denotes the set of all $w$-pseudo $\emph{S}$-asymptotically $(\xi,\tau,a)$-periodic functions.  Several special cases of the function set are distinguished by the following simplified notations: When $\xi(\lambda)=a^\uparrow(-\lambda)$, the set is denoted by ${\rm{PSAP}}_{\tau,a}(\mathbb{R},\mathbb{E},w)$. When $w(\lambda)\equiv 1$, the  set is denoted by ${\rm{PSAP}}_{\xi,\tau,a}(\mathbb{R},\mathbb{E})$. When $\xi(\lambda)=a^\uparrow(-\lambda)$, $w(\lambda)\equiv 1$, the set is simply written as ${\rm{PSAP}}_{\tau,a}(\mathbb{R},\mathbb{E})$.

Clearly, when the weight function $\xi(\lambda)=w(\lambda)\equiv1$ and $a=1$, we can obtain the space of pseudo $\emph{S}$-asymptotically periodic functions in \cite{Pierri}.

We have the following properties for the set ${\rm{PSAP}}_{\xi,\tau,a}(\mathbb{R},\mathbb{E},w)$.
\begin{Theorem}\label{dbdhb}
For any $w_i\in \mathcal{W}\ (i=1,2)$, if $w_1\sim w_2$, then the weighted function spaces coincide: $${\rm{PSAP}}_{\xi,\tau,a}(\mathbb{R},\mathbb{E},w_1)={\rm{PSAP}}_{\xi,\tau,a}(\mathbb{R},\mathbb{E},w_2).$$
\end{Theorem}
A direct consequence of Theorem \ref{dbdhb} is the following corollary.
\begin{Corollary}
If $w\in \mathcal{W}_b$, then ${\rm{PSAP}}_{\xi,\tau,a}(\mathbb{R},\mathbb{E},w)={\rm{PSAP}}_{\xi,\tau,a}(\mathbb{R},\mathbb{E})$.
\end{Corollary}

Combining Theorem 1 in \cite{Liangm7} and the fixed point technique in the space $BC_\xi(\mathbb{R},\mathbb{E})$, we can obtain the following result.
\begin{Theorem}
If $w_1\nsim w_2$, generally $${\rm{PSAP}}_{\xi,\tau,a}(\mathbb{R},\mathbb{E},w_1)\neq{\rm{PSAP}}_{\xi,\tau,a}(\mathbb{R},\mathbb{E},w_2).$$ In particular, for any $\tau\in \mathbb{R}$,
there exists sufficiently large $|a|$ and $w\notin \mathcal{W}_b$, such that $${\rm{PSAP}}_{\tau,a}(\mathbb{R},\mathbb{E},w)\neq{\rm{PSAP}}_{\tau,a}(\mathbb{R},\mathbb{E}).$$
\end{Theorem}

\begin{Lemma}\label{ddj}
For any $h_i\ (i=1,2)$, $h\in {\rm{PSAP}}_{\xi,\tau,a}(\mathbb{R},\mathbb{E},w)$, the assertions listed below are true:\\
(1) $h_1+h_2\in {\rm{PSAP}}_{\xi,\tau,a}(\mathbb{R},\mathbb{E},w)$, and, for any $c\in \mathbb{C}$, $ch\in {\rm{PSAP}}_{\xi,\tau,a}(\mathbb{R},\mathbb{E},w)$.\\
(2) Assume for each $r\in \mathbb{R}$, $w_r\sim w$ and $|\xi_r|\sim |\xi|$. Then, for any $r\in \mathbb{R}$, the function $h_r(t)\in {\rm{PSAP}}_{\xi,\tau,a}(\mathbb{R},\mathbb{E},w)$.
\end{Lemma}
\begin{proof}
(1) According to Definition \ref{hddj1}, if $h,\ h_i\ (i=1,2)\in {\rm{PSAP}}_{\xi,\tau,a}(\mathbb{R},\mathbb{E},w)$, then
$$\lim_{\ell\rightarrow \infty}\frac{1}{\sigma(\ell,w)}\int_{-\ell}^\ell \left\|\xi(\lambda)[h(\lambda+\tau)-ah(\lambda)]\right\|w(\lambda)d\lambda=0,$$
and
$$\lim_{\ell\rightarrow \infty}\frac{1}{\sigma(\ell,w)}\int_{-\ell}^\ell \left\|\xi(\lambda)[h_i(\lambda+\tau)-ah_i(\lambda)]\right\|w(\lambda)d\lambda=0,\ i=1,2.$$
So,
\begin{eqnarray*}
&&\lim_{\ell\rightarrow \infty}\frac{1}{\sigma(\ell,w)}\int_{-\ell}^\ell \left\|\xi(\lambda)[h_1(\lambda+\tau)+h_2(\lambda+\tau)-a(h_1(\lambda)+h_2(\lambda))]\right\|w(\lambda)d\lambda\\
&\leq&\lim_{\ell\rightarrow \infty}\frac{1}{\sigma(\ell,w)}\int_{-\ell}^\ell \left\|\xi(\lambda)[h_1(\lambda+\tau)-ah_1(\lambda)]\right\|w(\lambda)d\lambda\\
&&+\lim_{\ell\rightarrow \infty}\frac{1}{\sigma(\ell,w)}\int_{-\ell}^\ell \left\|\xi(\lambda)[h_2(\lambda+\tau)-ah_2(\lambda)]\right\|w(\lambda)d\lambda\\
&=&0,
\end{eqnarray*}
and for any $c\in \mathbb{C}$,
\begin{eqnarray*}
&&\lim_{\ell\rightarrow \infty}\frac{1}{\sigma(\ell,w)}\int_{-\ell}^\ell \left\|\xi(\lambda)[ch(\lambda+\tau)-ach(\lambda)]\right\|w(\lambda)d\lambda\\
&\leq&|c|\lim_{\ell\rightarrow \infty}\frac{1}{\sigma(\ell,w)}\int_{-\ell}^\ell \left\|\xi(\lambda)[h(\lambda+\tau)-ah(\lambda)]\right\|w(\lambda)d\lambda=0.
\end{eqnarray*}
From the definition, one can readily see that $h_1+h_2,\ ch\in {\rm{PSAP}}_{\xi,\tau,a}(\mathbb{R},\mathbb{E},w)$.

(2) Per Definition \ref{hddj1}, it follows that
$$\lim_{\ell\rightarrow \infty}\frac{1}{\sigma(\ell,w)}\int_{-\ell}^\ell \left\|\xi(\lambda)[h(\lambda+\tau)-ah(\lambda)]\right\|w(\lambda)d\lambda=0.$$
Therefore, for every $r\in \mathbb{R}$, the following equality holds:
\begin{eqnarray*}
&&\lim_{\ell\rightarrow \infty}\frac{1}{\sigma(\ell,w)}\int_{-\ell}^\ell \left\|\xi(\lambda)[h(\lambda+r+\tau)-ah(\lambda+r)]\right\|w(\lambda)d\lambda\\
&=&\lim_{\ell\rightarrow \infty}\frac{1}{\sigma(\ell,w)}\int_{-\ell+r}^{\ell+r} \left\|\xi(\lambda-r)[h(\lambda+\tau)-ah(\lambda)]\right\|w(\lambda-r)d\lambda.
\end{eqnarray*}
Given that $w_r\sim w$ and $|\xi_r|\sim |\xi|$ hold for all $r\in \mathbb{R}$, we can find a $C>0$ such that
\begin{center}
$\max\{w(\lambda-r),w(\lambda-|r|),w(\lambda+|r|)\} \leq Cw(\lambda)$, for a.e. $\lambda\in \mathbb{R}$,
\end{center}
and
\begin{center}
$|\xi(\lambda-r)| \leq C|\xi(\lambda)|$, for any $\lambda\in \mathbb{R}$.
\end{center}
So, we have
\begin{eqnarray*}
&&\lim_{\ell\rightarrow \infty}\frac{1}{\sigma(\ell,w)}\int_{-\ell}^\ell \left\|\xi(\lambda)[h(\lambda+r+\tau)-ah(\lambda+r)]\right\|w(\lambda)d\lambda\\
&\leq&\lim_{\ell\rightarrow \infty}\frac{C^2}{\sigma(\ell,w)}\int_{-\ell-|r|}^{\ell+|r|} \left\|\xi(\lambda)[h(\lambda+\tau)-ah(\lambda)]\right\|w(\lambda)d\lambda\\
&\leq&\lim_{\ell\rightarrow \infty}\frac{C^2\sigma(\ell+|r|,w)}{\sigma(\ell,w)}\frac{1}{\sigma(\ell+|r|,w)}\int_{-\ell-|r|}^{\ell+|r|} \left\|\xi(\lambda)[h(\lambda+\tau)-ah(\lambda)]\right\|w(\lambda)d\lambda\\
&\leq&\lim_{\ell\rightarrow \infty}\frac{2C^3}{\sigma(\ell+|r|,w)}\int_{-\ell-|r|}^{\ell+|r|} \left\|\xi(\lambda)[h(\lambda+\tau)-ah(\lambda)]\right\|w(\lambda)d\lambda,
\end{eqnarray*}
where in the last inequality above we use the fact
\begin{eqnarray*}
\sigma(\ell+|r|,w)&\leq& \int_{-\ell-|r|}^{\ell-|r|}w(\lambda)d\lambda+\int_{-\ell+|r|}^{\ell+|r|}w(\lambda)d\lambda,
\end{eqnarray*}
for sufficiently large $T$. Therefore, we get
$h_r(\lambda)\in {\rm{PSAP}}_{\xi,\tau,a}(\mathbb{R},\mathbb{E},w)$.
\end{proof}
\begin{Theorem}
Suppose for each $r\in \mathbb{R}$, $|\xi_r|\sim |\xi|$. Then ${\rm{PSAP}}_{\xi,\tau,a}(\mathbb{R},\mathbb{E},w)$ is  a Banach space endowed with the norm $\|\cdot\|_{b\xi}$.
\end{Theorem}
\begin{proof}
Lemma \ref{ddj}(1) ensures that ${\rm{PSAP}}_{\xi,\tau,a}(\mathbb{R},\mathbb{E},w)$ forms a vector space. To prove the closedness of the space
 ${\rm{PSAP}}_{\xi,\tau,a}(\mathbb{R},\mathbb{E},w)$, we select a sequence $$\{h_n\}_{n=1}^\infty\subseteq {\rm{PSAP}}_{\xi,\tau,a}(\mathbb{R},\mathbb{E}, w)$$ such that
$$\displaystyle{\lim_{n\rightarrow \infty}h_n=h\in BC_{\xi}(\mathbb{R},\mathbb{E})}.$$ Next we will prove that $h\in {\rm{PSAP}}_{\xi,\tau,a}(\mathbb{R},\mathbb{E},w)$.

Because $|\xi_\tau|\sim |\xi|$, there exists $C>0$ such that $$\left|\frac{\xi(\lambda)}{\xi(\lambda+\tau)}\right|\leq C,\qquad \lambda\in \mathbb{R}.$$
Given any $\mu>0$, one can find a positive integer $N$ and a constant $\ell^{\prime}>0$ such that
$$\|h_N-h\|_{b\xi}=\displaystyle{\sup_{\lambda\in \mathbb{R}}}\|\xi(\lambda)[h_N(\lambda)-h(\lambda)]\|\leq \frac{\mu}{3}\min\left\{\frac{1}{C}, \frac{1}{|a|}\right\},$$
and for any $\ell>\ell^{\prime}$,
$$\frac{1}{\sigma(T,w)}\int_{-\ell}^\ell \left\|\xi(\lambda)[h_N(\lambda+\tau)-ah_N(\lambda)]\right\|w(\lambda)d\lambda\leq \frac{\mu}{3}.$$
Then for any $\ell>\ell^{\prime}$, we have
\begin{eqnarray*}
&&\frac{1}{\sigma(\ell,w)}\int_{-\ell}^\ell \left\|\xi(\lambda)[h(\lambda+\tau)-ah(\lambda)]\right\|w(\lambda)d\lambda\\
&=&\frac{1}{\sigma(\ell,w)}\int_{-\ell}^\ell \left\|\xi(\lambda)[h(\lambda+\tau)-h_N(\lambda+\tau)+h_N(\lambda+\tau)-ah_N(\lambda)+ah_N(\lambda)-ah(\lambda)]\right\|\\
&&\times w(\lambda)d\lambda\\
&\leq&\frac{C}{\sigma(\ell,w)}\int_{-\ell}^\ell \left\|\xi(\lambda+\tau)[h(\lambda+\tau)-h_N(\lambda+\tau)]\right\|w(\lambda)d\lambda\\
&&+\frac{1}{\sigma(\ell,w)}\int_{-\ell}^\ell \left\|\xi(\lambda)[h_N(\lambda+\tau)-ah_N(\lambda)]\right\|w(\lambda)d\lambda\\
&&+\frac{1}{\sigma(\ell,w)}\int_{-\ell}^\ell |a|\left\|\xi(\lambda)[h_N(\lambda)-h(\lambda)]\right\|w(\lambda)d\lambda\\
&\leq&C\|h_N-h\|_{b\xi}+\frac{1}{\sigma(\ell,w)}\int_{-\ell}^\ell \left\|\xi(\lambda)[h_N(\lambda+\tau)-ah_N(\lambda)]\right\|w(\lambda)dt+|a|\|h_N-h\|_{b\xi}\\
&\leq& \frac{\mu}{3}+\frac{\mu}{3}+\frac{\mu}{3}=\mu,
\end{eqnarray*}
which implies that $h\in {\rm{PSAP}}_{\xi,\tau,a}(\mathbb{R},\mathbb{E},w)$. This concludes the proof of the theorem.
\end{proof}

Clearly, when $\xi(\cdot)a^\uparrow(\cdot)\in BC(\mathbb{R},\mathbb{C})$ and $w\in \mathcal{W}_b$, $$(\mathcal{P}_{\tau,a}(\mathbb{R},\mathbb{E}),\|\cdot\|_{b\xi})\subseteq({\rm{SAP}}_{\xi,\tau,a}(\mathbb{R},\mathbb{E}),\|\cdot\|_{b\xi})\subseteq ({\rm{PSAP}}_{\xi,\tau,a}(\mathbb{R},\mathbb{E},w),\|\cdot\|_{b\xi}).$$

Fig. \ref{fig:1} shows the graphs of $h_1(\lambda)=2^\lambda {\rm{sin}(2\pi \lambda)}\in \mathcal{P}_{1,2}(\mathbb{R},\mathbb{R})$. Fig. \ref{fig:2} and Fig. \ref{fig:4} show the graphs of $h_2(\lambda)=\frac{2^\lambda {\rm{sin}(2\pi \lambda)}}{1+\lambda^2}\in {\rm{SAP}}_{1,2}(\mathbb{R},\mathbb{R})$ and $h_3(\lambda)=\frac{2^\lambda {\rm{sin}(2\pi \lambda)}}{1+\lambda^4}\in {\rm{PSAP}}_{1,2}(\mathbb{R},\mathbb{R},w)$ with $w(\lambda)=1+\lambda^2$, which are verified by the graphs of $h_2^{\prime}(\lambda)$ and $h_3^{\prime}(\lambda)$ in Fig. \ref{fig:3} and Fig. \ref{fig:5}, respectively.
\begin{figure}[h!]
\centering
  {
    \raisebox{0.08\height}{
     \begin{minipage}{7cm}
            \includegraphics[height=4.75cm,width=6.5cm]{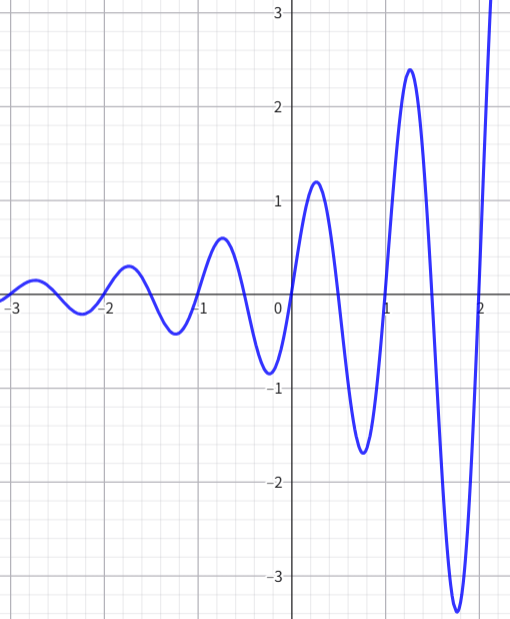}
            \caption{$h_1(\lambda)=2^\lambda {\rm{sin}(2\pi \lambda)}$}
            \label{fig:1}
       \end{minipage}}
   }
\end{figure}
\begin{figure}[h!]
\centering
  {
    \raisebox{0.08\height}{
     \begin{minipage}{7cm}
            \includegraphics[height=4.75cm,width=6.5cm]{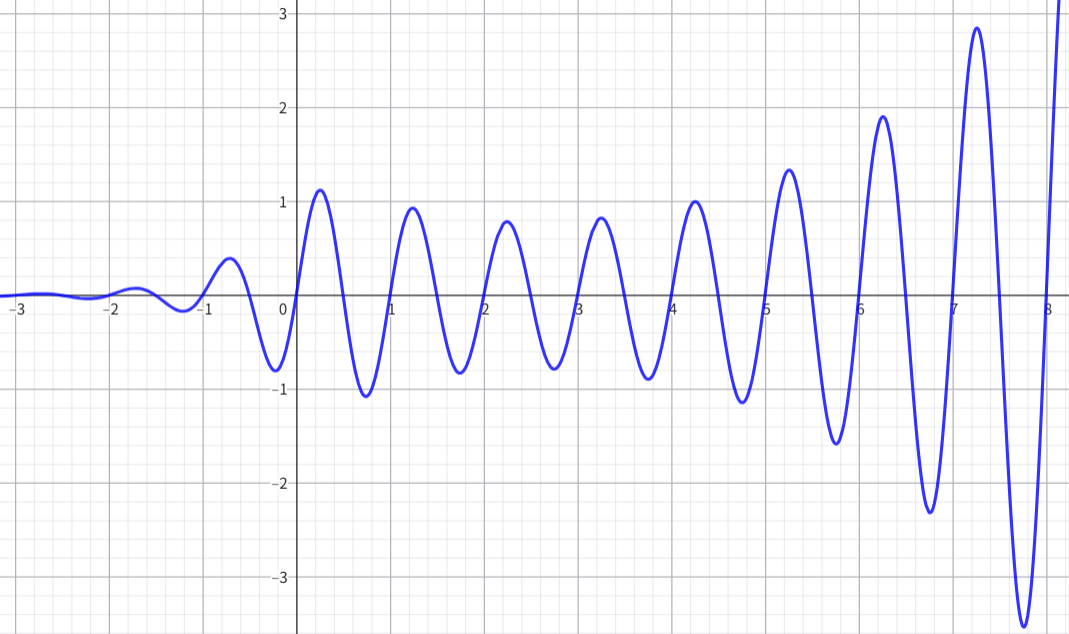}
            \caption{$h_2(\lambda)=\frac{2^\lambda {\rm{sin}(2\pi \lambda)}}{1+\lambda^2}$}
            \label{fig:2}
       \end{minipage}}
       \begin{minipage}{7cm}
            \includegraphics[height=4.3cm,width=6.5cm]{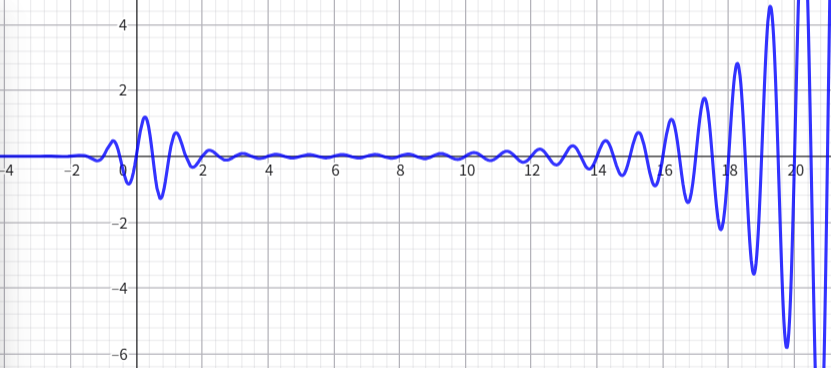}
            \caption{$h_3(\lambda)=\frac{2^\lambda {\rm{sin}(2\pi \lambda)}}{1+\lambda^4}$}
            \label{fig:3}
       \end{minipage}
   }
\end{figure}
\begin{figure}[h!]
\centering
  {
    \raisebox{0.08\height}{
     \begin{minipage}{7cm}
            \includegraphics[height=2.1cm,width=6.5cm]{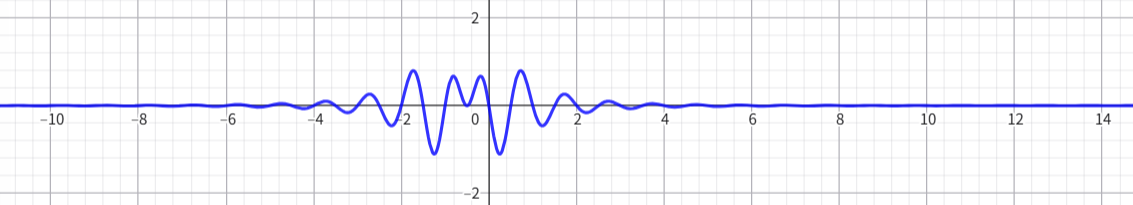}
            \caption{$h_2^{\prime}(\lambda)=2^{-\lambda}[h_2(\lambda+1)-2h_2(\lambda)]$}
            \label{fig:4}
       \end{minipage}}
     \begin{minipage}{7cm}
            \includegraphics[height=1.5cm,width=6.5cm]{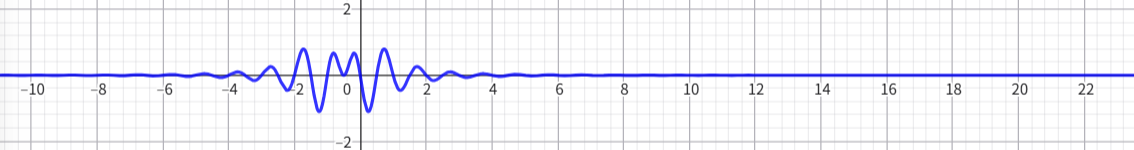}
            \caption{$h_3^{\prime}(\lambda)=\frac{1}{\sigma(\lambda,w)}\int_{-\lambda}^\lambda 2^{-\theta}[h_3(\theta+1)-2h_3(\theta)]w(\theta)d\theta$ with $w(\lambda)=1+\lambda^2$}
            \label{fig:5}
       \end{minipage}
   }
\end{figure}

Next, we prove a convolution  theorem. Before that, let us introduce the following set
\begin{center}
$\widetilde{\mathbb{U}}:=\Big\{\xi(\lambda)\in C(\mathbb{R},\mathbb{C}):\ \xi\neq 0$, there exists a positive constant $\mathcal{M}$ such that $ \left|\frac{\xi(\imath+\jmath)}{\xi(\imath)\xi(\jmath)}\right|\leq \mathcal{M}$ for any $\imath,\ \jmath\in \mathbb{R}\Big\}$.
\end{center}
\begin{Theorem}
Suppose $\{\mathcal{T}(\lambda)\}_{\lambda\geq 0}$ is a family of bounded operators on $\mathbb{E}$ satisfying that $\{\xi(\lambda)\mathcal{T}(\lambda)\}_{\lambda\geq 0}$ is strongly continuous and $\xi(\lambda)\mathcal{T}(\lambda)\in
L^1([0,\infty),\mathcal{B}(\mathbb{E}))$. Also suppose  $\xi\in \widetilde{\mathbb{U}}$, $w_r\sim w$ and $|\xi_r|\sim |\xi|$ for any $r\in \mathbb{R}$. Then for any $h\in {\rm{PSAP}}_{\xi,\tau,a}(\mathbb{R},\mathbb{E},w)$, the convolution
$$\phi(\lambda):=\int_{-\infty}^\lambda \mathcal{T}(\lambda-\theta)h(\theta)d\theta$$
also belongs to ${\rm{PSAP}}_{\xi,\tau,a}(\mathbb{R},\mathbb{E},w)$.
\end{Theorem}
\begin{proof}
Since $\xi(\lambda)\mathcal{T}(\lambda)\in L^1([0,\infty),\mathcal{B}(\mathbb{E}))$ and $h\in {\rm{PSAP}}_{\xi,\tau,a}(\mathbb{R},\mathbb{E},w)$, we get
$$\int_0^\infty \|\xi(\lambda)\mathcal{T}(\lambda)\|dt< \infty,\ h(\lambda)=\frac{1}{\xi(\lambda)}h_1(\lambda),\ h_1(\lambda)\in BC(\mathbb{R},\mathbb{E}).$$
Then the fact $\xi\in \widetilde{\mathbb{U}}$ leads to
\begin{eqnarray*}
\|\phi(\lambda)\|&=&\left\|\int_{-\infty}^\lambda \mathcal{T}(\lambda-\theta)h(\theta)d\theta\right\|=\left\|\int_0^\infty \frac{1}{\xi(\lambda-\theta)}\mathcal{T}(\theta)h_1(\lambda-\theta)d\theta\right\|\\
&\leq&\int_0^\infty \left|\frac{1}{\xi(\lambda-\theta)}\right|\cdot\|\mathcal{T}(\theta)\|\|h_1(\lambda-\theta)\|d\theta\\
&\leq&\int_0^\infty \mathcal{M}\left|\frac{\xi(\theta)}{\xi(\lambda)}\right|\cdot\|\mathcal{T}(\theta)\|\|h_1(\lambda-\theta)\|d\theta,
\end{eqnarray*}
and
\begin{eqnarray*}
&&\int_0^\infty |\xi(\theta)|\|\mathcal{T}(\theta)\|\|h_1(\lambda-\theta)\|d\theta\leq  \|h_1\|_{\infty}\int_0^\infty \left \|\xi(\theta)\mathcal{T}(\theta)\right\|d\theta<\infty,
\end{eqnarray*}
which means that $\displaystyle{\int_0^\infty \xi(\theta)\mathcal{T}(\theta)h_1(\lambda-\theta)ds\in BC(\mathbb{R},\mathbb{E})}$ and $\phi(\lambda)\in BC_{\xi}(\mathbb{R},\mathbb{E})$. By using the Fubini theorem and the fact $\xi\in \widetilde{\mathbb{U}}$, we have
\begin{eqnarray*}
&&\frac{1}{\sigma(\ell,w)}\int_{-\ell}^\ell \left\|\xi(\lambda)[\phi(\lambda+\tau)-a\phi(\lambda)]\right\|w(\lambda)d\lambda\\
&=&\frac{1}{\sigma(\ell,w)}\int_{-\ell}^\ell \left\|\xi(\lambda)\left[\int_{-\infty}^{\lambda+\tau} \mathcal{T}(\lambda+\tau-\theta)h(\theta)ds-a\int_{-\infty}^\lambda \mathcal{T}(\lambda-\theta)h(\theta)d\theta\right]\right\|w(\lambda)d\lambda\\
&=&\frac{1}{\sigma(\ell,w)}\int_{-\ell}^\ell \left\|\xi(\lambda)\left[\int_{-\infty}^{\lambda} \mathcal{T}(\lambda-\theta)h(\theta+\tau)d\theta-a\int_{-\infty}^\lambda \mathcal{T}(\lambda-\theta)h(\theta)ds\right]\right\|w(\lambda)d\lambda\\
&=&\frac{1}{\sigma(\ell,w)}\int_{-\ell}^\ell \left\|\xi(\lambda)\int_0^\infty \mathcal{T}(\theta)[h(\lambda-\theta+\tau)-ah(\lambda-\theta)]d\theta\right\|w(\lambda)d\lambda\\
&\leq&\frac{1}{\sigma(\ell,w)}\int_{-\ell}^\ell \int_0^\infty \left\|\xi(\lambda)\mathcal{T}(\theta)[h(\lambda-\theta+\tau)-ah(\lambda-\theta)]\right\|w(\lambda)d\theta d\lambda\\
&\leq&\mathcal{M}\int_0^\infty \|\xi(\theta)\mathcal{T}(\theta)\|\left[\frac{1}{\sigma(\ell,w)}\int_{-\ell}^\ell  \left\|\xi(\lambda-\theta)[h(\lambda-\theta+\tau)-ah(\lambda-\theta)]\right\|w(\lambda)d\lambda\right]d\theta.
\end{eqnarray*}
Applying the Lebesgue's dominated convergence theorem, the fact $h\in {\rm{PSAP}}_{\xi,\tau,a}(\mathbb{R},\mathbb{E},w)$ and Lemma \ref{ddj}(2), we can deduce
$$\lim_{\ell\rightarrow \infty}\frac{1}{\sigma(\ell,w)}\int_{-\ell}^\ell \left\|\xi(\lambda)[\phi(\lambda+\tau)-a\phi(\lambda)]\right\|w(\lambda)dt=0,$$
i.e. $\phi\in {\rm{PSAP}}_{\xi,\tau,a}(\mathbb{R},\mathbb{E},w).$ We have thus finished the proof of this theorem.
\end{proof}
\begin{Example}
Obviously, $\xi(t)=a^\uparrow(t)\in \widetilde{\mathbb{U}}$. And we can also directly verify that
$$\xi(t)=a^\uparrow(-t)+a^\uparrow(t)\in \widetilde{\mathbb{U}}$$
with $\mathcal{M}=1$.

Next we give another example.

For any $a>1,\ \tau>0$, let $\xi(t)=a^\uparrow(-t)+\xi_1(t)$, where
\[\xi_1(t)=
\left\{
\begin{array}{lr}
1,\ \ \qquad \ t< 0,\\
1-t,\ \quad t\in [0,1],\\
0,\ \ \qquad \ t> 1.
\end{array}
\right.
\]
Clearly, $\xi\neq 0$, $\xi(\cdot)a^\uparrow(\cdot)\in BC(\mathbb{R},\mathbb{C})$ and $|\xi_r|\sim |\xi|$ for any $r\in \mathbb{R}$. And we can compute
\begin{eqnarray*}
&&\left|\frac{\xi(t+s)}{\xi(t)\xi(s)}\right|=\frac{a^\uparrow(-t-s)+\xi_1(t+s)}{a^\uparrow(-t-s)+[a^\uparrow(-t)\xi_1(s)+a^\uparrow(-s)\xi_1(t)]+\xi_1(t)\xi_1(s)}.
\end{eqnarray*}
When $t+s\leq 1$,
\begin{eqnarray*}
&&\left|\frac{\xi(t+s)}{\xi(t)\xi(s)}\right|\leq \frac{a^\uparrow(-t-s)+1}{a^\uparrow(-t-s)}=1+a^\uparrow(t+s)\leq 1+a^\uparrow(1);
\end{eqnarray*}
when $t+s>1$,
\begin{eqnarray*}
&&\left|\frac{\xi(t+s)}{\xi(t)\xi(s)}\right|\leq\frac{a^\uparrow(-t-s)}{a^\uparrow(-t-s)}=1.
\end{eqnarray*}
So $\xi\in \widetilde{\mathbb{U}}$ with $\mathcal{M}=1+a^\uparrow(1)$.

However, $\xi(t)=(a^\uparrow(t)+t^2)^{-1}\notin \widetilde{\mathbb{U}}$. Actually, we have
\begin{eqnarray*}
&&\left|\frac{\xi(t+s)}{\xi(t)\xi(s)}\right|=\left|\frac{a^\uparrow(t+s)+t^2a^\uparrow(s)+s^2a^\uparrow(t)+t^2s^2}{a^\uparrow(t+s)+(t+s)^2}\right|\\
&\overset{s=-t}{=}& \left|1+t^2a^\uparrow(t)+t^2a^\uparrow(-t)+t^4\right|.
\end{eqnarray*}
When $|a|>1,\ \tau>0,\ t>0$, we have
\begin{eqnarray*}
&&\left|1+t^2a^\uparrow(t)+t^2a^\uparrow(-t)+t^4\right|\geq t^2|a|^\uparrow(t)-\left(1+t^4+t^2|a|^\uparrow(-t)\right)\rightarrow +\infty \ (t\rightarrow +\infty).
\end{eqnarray*}
In other cases, we also have similar results. So we have $\xi(t)=(a^\uparrow(t)+t^2)^{-1}\notin \widetilde{\mathbb{U}}$.
\end{Example}

In the sequel, we will give a composition theorem on weighted pseudo $\emph{S}$-asymptotically $(\xi,\tau,a)$-periodic functions. First, let us denote the Nemytskil operator corresponding to the function $h\in C(\mathbb{R}\times \mathbb{F},\mathbb{E})$ by:
$$\overline{\Psi}_h:\ {\rm{PSAP}}_{\xi,\tau,a}(\mathbb{R},\mathbb{F},w)\rightarrow C(\mathbb{R},\mathbb{E}),\qquad \overline{\Psi}_h(\zeta)(\lambda)=h(\lambda,\zeta(\lambda)).$$
\begin{Theorem}\label{djdk}
Let $h\in C(\mathbb{R}\times \mathbb{F},\mathbb{E})$ satisfy the following conditions:\\
$(I_1)$ $(i)$ For any $y\in \mathbb{F}$, $h(\cdot,y)\in BC_{\xi}(\mathbb{R},\mathbb{E})$, and $\displaystyle{\sup_{ y\in \mathbb{Y}}\left\|h(\cdot,y)\right\|_{b\xi}<\infty}$;\\
$(ii)$ Given the limit $\displaystyle{\lim_{\ell\rightarrow \infty}\frac{1}{\sigma(\ell,w)}\int_{-\ell}^\ell \left\|\xi(\lambda)[h(\lambda+\tau,ay)-ah(\lambda,y)]\right\|w(\lambda)d\lambda}$, it converges to zero uniformly with respect to $y\in \mathbb{F}$.\\
$(I_2)$ There exists a positive constant $K$ such that
$$\|h(\lambda,y_1)-h(\lambda,y_2)\|\leq K\|y_1-y_2\|_{\mathbb{F}},\ \forall   y_i\in \mathbb{F}\ (i=1,2),\ \lambda\in\mathbb{R}.$$
Then for any $\zeta\in {\rm{PSAP}}_{\xi,\tau,a}(\mathbb{R},\mathbb{F},w)$, we have $\overline{\Psi}_h(\zeta)\in {\rm{PSAP}}_{\xi,\tau,a}(\mathbb{R},\mathbb{E},w)$.
\end{Theorem}
\begin{proof}
By $(I_1)(i)$ we have $\overline{\Psi}_h(\zeta)\in BC_{\xi}(\mathbb{R},\mathbb{E})$. For any $\zeta\in {\rm{PSAP}}_{\xi,\tau,a}(\mathbb{R},\mathbb{F},w)$,
we get
$$\lim_{\ell\rightarrow \infty}\frac{1}{\sigma(\ell,w)}\int_{-\ell}^\ell \left\|\xi(\lambda)[\zeta(\lambda+\tau)-a\zeta(\lambda)]\right\|_{\mathbb{F}}w(\lambda)d\lambda=0.$$
And we have
\begin{eqnarray*}
&&\frac{1}{\sigma(\ell,w)}\int_{-\ell}^\ell \left\|\xi(\lambda)[\overline{\Psi}_h(\zeta)(\lambda+\tau)-a\overline{\Psi}_h(\zeta)(\lambda)]\right\|w(\lambda)d\lambda\\
&=&\frac{1}{\sigma(\ell,w)}\int_{-\ell}^\ell \Bigg\|\xi(\lambda)\Bigg[h(\lambda+\tau,\zeta(\lambda+\tau))-ah\left(\lambda,\frac{1}{a}\zeta(\lambda+\tau)\right)\\
&&+ah\left(\lambda,\frac{1}{a}\zeta(\lambda+\tau)\right)-ah(\lambda,\zeta(\lambda))\Bigg]\Bigg\|w(\lambda)d\lambda\\
&\leq& \frac{1}{\sigma(\ell,w)}\int_{-\ell}^\ell \Bigg\|\xi(\lambda)\Bigg[h(\lambda+\tau,\zeta(\lambda+\tau))-ah\left(\lambda,\frac{1}{a}\zeta(\lambda+\tau)\right)\Bigg]\Bigg\|w(\lambda)d\lambda\\
&&+\frac{1}{\sigma(\ell,w)}\int_{-\ell}^\ell \Bigg\|\xi(\lambda)\Bigg[ah\left(\lambda,\frac{1}{a}\zeta(\lambda+\tau)\right)-ah(\lambda,\zeta(\lambda))\Bigg]\Bigg\|w(\lambda)d\lambda\\
&:=&\mathbb{J}_1+\mathbb{J}_2.
\end{eqnarray*}
From $(I_1)(ii)$, we can see $\mathbb{J}_1\rightarrow 0$ as $\ell\rightarrow \infty$. For $\mathbb{J}_2$, by $(I_2)$ we have
\begin{eqnarray*}
&&\lim_{\ell\rightarrow \infty}\frac{1}{\sigma(\ell,w)}\int_{-\ell}^\ell \Bigg\|\xi(\lambda)\Bigg[ah\left(\lambda,\frac{1}{a}\zeta(\lambda+\tau)\right)-ah(\lambda,\zeta(\lambda))\Bigg]\Bigg\|w(\lambda)d\lambda\\
&\leq& \lim_{\ell\rightarrow \infty}\frac{1}{\sigma(\ell,w)}\int_{-\ell}^\ell \left|\xi(\lambda)\right||a|K\left\|\frac{1}{a}\zeta(\lambda+\tau)-\zeta(\lambda)\right\|_{\mathbb{F}}w(\lambda)d\lambda\\
&=& K\lim_{T\rightarrow \infty}\frac{1}{\sigma(\ell,w)}\int_{-\ell}^\ell \left|\xi(\lambda)\right|\left\|\zeta(\lambda+\tau)-a\zeta(\lambda)\right\|_{\mathbb{F}}w(\lambda)d\lambda=0.
\end{eqnarray*}
Thus, we get
 $$\lim_{\ell\rightarrow \infty}\frac{1}{\sigma(\ell,w)}\int_{-\ell}^\ell \left\|\xi(\lambda)[\overline{\Psi}_h(\zeta)(\lambda+\tau)-a\overline{\Psi}_h(\zeta)(\lambda)]\right\|w(\lambda)d\lambda=0,$$
 which is the desired result.
\end{proof}
 For any $h\in BC_{\xi}(\mathbb{R},\mathbb{E})$, $\ell,\ \mu>0$, $\Sigma_{\ell,\mu}(h)$
denotes the set $$\left\{\lambda\in [-\ell,\ell]:\ \|\xi(\lambda)[h(\lambda+\tau)-ah(\lambda)]\|\geq \mu\right\}.$$
\begin{Lemma}\label{jdkkd}
Assume $|\xi_r|\sim |\xi|$ for any $r\in \mathbb{R}$. $h\in {\rm{PSAP}}_{\xi,\tau,a}(\mathbb{R},\mathbb{E},w)$ if and only if
for any fixed $\mu>0$, $$\displaystyle{\lim_{\ell\rightarrow \infty}}\frac{1}{\sigma(\ell,w)}\int_{\Sigma_{\ell,\mu}(h)}w(\lambda)d\lambda=0.$$
\end{Lemma}
\begin{proof}
(Necessity) From
\begin{eqnarray*}
&&\frac{1}{\sigma(\ell,w)}\int_{-\ell}^\ell \left\|\xi(\lambda)[h(\lambda+\tau)-ah(\lambda)]\right\|w(\lambda)d\lambda\\
&=& \frac{1}{\sigma(\ell,w)}\int_{[-\ell,\ell]\backslash \Sigma_{\ell,\mu}(h)} \left\|\xi(\lambda)[h(\lambda+\tau)-ah(\lambda)]\right\|w(\lambda)d\lambda\\
&&+\frac{1}{\sigma(\ell,w)}\int_{ \Sigma_{\ell,\mu}(h)} \left\|\xi(\lambda)[h(\lambda+\tau)-ah(\lambda)]\right\|w(\lambda)d\lambda\\
&\geq&\frac{1}{\sigma(\ell,w)}\int_{ \Sigma_{\ell,\mu}(h)} \left\|\xi(\lambda)[h(\lambda+\tau)-ah(\lambda)]\right\|w(\lambda)d\lambda\\
&\geq&\frac{\mu}{\sigma(\ell,w)}\int_{ \Sigma_{\ell,\mu}(h)} w(\lambda)d\lambda,
\end{eqnarray*}
we can see that necessity is established.

(Sufficiency)
Because $|\xi_\tau|\sim |\xi|$, there exists $C>0$ such that $\left|\frac{\xi(\lambda)}{\xi(\lambda+\tau)}\right|\leq C,\ \lambda\in \mathbb{R}$.

Since
\begin{eqnarray*}
&&\frac{1}{\sigma(\ell,w)}\int_{-\ell}^\ell \left\|\xi(\lambda)[h(\lambda+\tau)-ah(\lambda)]\right\|w(\lambda)d\lambda\\
&=& \frac{1}{\sigma(\ell,w)}\int_{[-\ell,\ell]\backslash \Sigma_{\ell,\mu}(h)} \left\|\xi(\lambda)[h(\lambda+\tau)-ah(\lambda)]\right\|w(\lambda)d\lambda\\
&&+\frac{1}{\sigma(\ell,w)}\int_{ \Sigma_{\ell,\mu}(h)} \left\|\xi(\lambda)[h(\lambda+\tau)-ah(\lambda)]\right\|w(\lambda)d\lambda\\
&\leq&\left(1-\frac{1}{\sigma(\ell,w)}\int_{\Sigma_{\ell,\mu}(h)}w(\lambda)d\lambda\right)\mu+\frac{C+|a|}{\sigma(\ell,w)}\|h\|_{b\xi}\int_{\Sigma_{\ell,\mu}(h)}w(\lambda)d\lambda,
\end{eqnarray*}
we can verify the correctness of sufficiency.

The proof of this lemma is thus finished.
\end{proof}

Next we present some other composition results for functions $h\in {\rm{PSAP}}_{\xi,\tau,a}(\mathbb{R},\mathbb{E},w)$ by using Lemma \ref{jdkkd}.
For any bounded subset $\Theta$ in $\mathbb{F}$, a function $h(\lambda,y)\in C(\mathbb{R}\times \mathbb{F},\mathbb{E})$ is said to
be uniformly continuous in $\Theta$ uniformly with respect to $\lambda$
if for any $\mu>0$, there exists a $\mathbbm{d}>0$ such that
for any $y_i\in \Theta\ (i=1,2)$ satisfying $\|y_1-y_2\|_{\mathbb{F}}<\mathbbm{d}$, we have
$$\|h(\lambda,y_1)-h(\lambda,y_2)\|\leq \mu,\  \lambda\in \mathbb{R}.$$
Let
\begin{equation}\label{bdjkdb}
\mathcal{H}(\lambda,y):=\xi(\lambda)h\left(\lambda,\frac{1}{\xi(\lambda)}y\right).
\end{equation}
\begin{Theorem}\label{dbdnj}
Let $h\in C(\mathbb{R}\times \mathbb{F},\mathbb{E})$ satisfy $(I_1)$ and the following condition:\\
$(I_2^{\prime})$ For any bounded subset $\Theta$ in $\mathbb{F}$, the function $\mathcal{H}(\lambda,y)$ is uniformly continuous in $\Theta$ uniformly with respect to $\lambda$.
Then for any $\zeta\in {\rm{PSAP}}_{\xi,\tau,a}(\mathbb{R},\mathbb{F},w)$, we have $\overline{\Psi}_h(\zeta)\in {\rm{PSAP}}_{\xi,\tau,a}(\mathbb{R},\mathbb{E},w)$.
\end{Theorem}
\begin{proof}
It follows from $(I_1)(i)$ that $\overline{\Psi}_h(\zeta)\in BC_{\xi}(\mathbb{R},\mathbb{E})$, and there exists a positive constant $C$ such that
$\displaystyle{\sup_{\lambda\in \mathbb{R}}\left\|\xi(\lambda)h(\lambda,y)\right\|\leq C}, \ \forall y\in \mathbb{F}$. On the other hand, we have
\begin{eqnarray*}
&&\frac{1}{\sigma(\ell,w)}\int_{-\ell}^\ell \left\|\xi(\lambda)[\overline{\Psi}_h(\zeta)(\lambda+\tau)-a\overline{\Psi}_h(\zeta)(\lambda)]\right\|w(\lambda)d\lambda\\
&\leq& \frac{1}{\sigma(\ell,w)}\int_{-\ell}^\ell \Bigg\|\xi(\lambda)\Bigg[h(\lambda+\tau,\zeta(\lambda+\tau))-ah\left(\lambda,\frac{1}{a}\zeta(\lambda+\tau)\right)\Bigg]\Bigg\|w(\lambda)d\lambda\\
&&+\frac{1}{\sigma(\ell,w)}\int_{-\ell}^\ell \Bigg\|\xi(\lambda)\Bigg[ah\left(\lambda,\frac{1}{a}\zeta(\lambda+\tau)\right)-ah(\lambda,\zeta(\lambda))\Bigg]\Bigg\|w(\lambda)d\lambda\\
&:=&\mathbb{J}_1+\mathbb{J}_2.
\end{eqnarray*}
$(I_1)(ii)$ implies that $\mathbb{J}_1\rightarrow 0$ as $\ell\rightarrow \infty$. Next, let us consider the bounded subset $$\Theta:=\{\xi(\lambda)\zeta(\lambda):\ \lambda\in \mathbb{R}\}$$
in $\mathbb{F}$.  For any given $\mu>0$, denote by $\mathbbm{d}$ the positive constant determined by the uniform continuity of the function $\mathcal{H}(\lambda,y)$ from $(I_2^{\prime})$ on $\Theta,$ which depends on $\mu$ but is independent of $\lambda$. Note that when $$\left\|\xi(\lambda)[\zeta(\lambda+\tau)-a\zeta(\lambda)]\right\|_{\mathbb{F}}<|a|\mathbbm{d},$$ we have
$$\left\|\frac{\xi(\lambda)}{a}\zeta(\lambda+\tau)-\xi(\lambda)\zeta(\lambda)\right\|_{\mathbb{F}}=\left\|\frac{1}{a}\xi(\lambda)[\zeta(\lambda+\tau)-a\zeta(\lambda)]\right\|_{\mathbb{F}}<\mathbbm{d},$$
which leads to
$$\left\|\xi(\lambda)\left[h\left(\lambda,\frac{1}{\xi(\lambda)}\frac{\xi(\lambda)}{a}\zeta(\lambda+\tau)\right)-h\left(\lambda,\frac{1}{\xi(\lambda)}\xi(\lambda)\zeta(\lambda)\right)\right]\right\|\leq\mu.$$
Hence, it follows that
\begin{eqnarray}
&&\frac{1}{\sigma(\ell,w)}\int_{-\ell}^\ell \Bigg\|\xi(\lambda)\Bigg[ah\left(\lambda,\frac{1}{a}\zeta(\lambda+\tau)\right)-ah(\lambda,\zeta(\lambda))\Bigg]\Bigg\|w(\lambda)d\lambda\nonumber\\
&=&\frac{|a|}{\sigma(\ell,w)}\int_{-\ell}^\ell\left\|\xi(\lambda)\left[h\left(\lambda,\frac{1}{\xi(\lambda)}\frac{\xi(\lambda)}{a}\zeta(\lambda+\tau)\right)
-h\left(\lambda,\frac{1}{\xi(\lambda)}\xi(\lambda)\zeta(\lambda)\right)\right]
\right\| w(\lambda)d\lambda\nonumber\\
&=&\frac{|a|}{\sigma(\ell,w)}\int_{[-\ell,\ell]\backslash\Sigma_{\ell,|a|\mathbbm{d}}(\zeta)}\left\|\xi(\lambda)\left[h\left(\lambda,\frac{1}{\xi(\lambda)}
\frac{\xi(\lambda)}{a}\zeta(\lambda+\tau)\right)-h\left(\lambda,\frac{1}{\xi(\lambda)}\xi(\lambda)\zeta(\lambda)\right)\right]\right\| w(\lambda)d\lambda\nonumber\\
&&+\frac{|a|}{\sigma(\ell,w)}\int_{\Sigma_{\ell,|a|\mathbbm{d}}(\zeta)}\left\|\xi(\lambda)\left[h\left(\lambda,\frac{1}{\xi(\lambda)}\frac{\xi(\lambda)}{a}
\zeta(\lambda+\tau)\right)-h\left(\lambda,\frac{1}{\xi(\lambda)}\xi(\lambda)\zeta(\lambda)\right)\right]
\right\|w(\lambda)d\lambda\nonumber\\
&\leq&\left(1-\frac{1}{\sigma(\ell,w)}\int_{\Sigma_{\ell,|a|\mathbbm{d}}(\zeta)}w(\lambda)d\lambda\right)|a|\mu
+2|a|C\frac{1}{\sigma(\ell,w)}\int_{\Sigma_{\ell,|a|\mathbbm{d}}(\zeta)}w(\lambda)d\lambda.\label{dhdjk}
\end{eqnarray}
From Lemma \ref{jdkkd}, the condition $\zeta\in {\rm{PSAP}}_{\xi,\tau,a}(\mathbb{R},\mathbb{F},w)$, and the fact that $\mu$ is arbitrary, we can infer from $(\ref{dhdjk})$ that
$\mathbb{J}_2$ also tends to $0$ as $\ell\rightarrow \infty$. Therefore, we obtain
$$\lim_{\ell\rightarrow \infty}\frac{1}{\sigma(\ell,w)}\int_{-\ell}^\ell \left\|\xi(\lambda)[\overline{\Psi}_h(\zeta)(\lambda+\tau)-a\overline{\Psi}_h(\zeta)(\lambda)]\right\|w(\lambda)d\lambda=0,$$
that is, $\overline{\Psi}_h(\zeta)\in {\rm{PSAP}}_{\xi,\tau,a}(\mathbb{R},\mathbb{E},w)$. The proof of this theorem is complete.
\end{proof}

\begin{Theorem}\label{dndk}
Let $h\in C(\mathbb{R}\times \mathbb{F},\mathbb{E})$ satisfy $(I_1)$ and the following condition:\\
$(I_2^{\prime\prime})$ There exists a function $G(\cdot)\in \mathcal{B}_w$ satisfying that for every $\mu>0$, there exists a $\mathbbm{d}>0$ such that
for all $y_i\in \mathbb{F}\ (i=1,2)$ satisfying $\|y_1-y_2\|_{\mathbb{F}}<\mathbbm{d}$, we have
$$\|\mathcal{H}(\lambda,y_1)-\mathcal{H}(\lambda,y_2)\|\leq G(\lambda)\mu,\ \lambda\in \mathbb{R},$$
where $\mathcal{H}$ is defined by formula $(\ref{bdjkdb})$. Then for any $\zeta\in {\rm{PSAP}}_{\xi,\tau,a}(\mathbb{R},\mathbb{F},w)$, we have $\overline{\Psi}_h(\zeta)\in {\rm{PSAP}}_{\xi,\tau,a}(\mathbb{R},\mathbb{E},w)$.
\end{Theorem}
\begin{proof}
It follows from $(I_1)(i)$ that $\overline{\Psi}_h(\zeta)\in BC_{\xi}(\mathbb{R},\mathbb{E})$, and there exists a positive constant $C$ such that
$$\displaystyle{\sup_{\lambda\in \mathbb{R}}\left\|\xi(\lambda)h(\lambda,y)\right\|\leq C}, \quad \forall  y\in \mathbb{F}.$$ Meanwhile, we also have
\begin{eqnarray*}
&&\frac{1}{\sigma(\ell,w)}\int_{-\ell}^\ell \left\|\xi(\lambda)[\overline{\Psi}_h(\zeta)(\lambda+\tau)-a\overline{\Psi}_h(\zeta)(\lambda)]\right\|w(\lambda)d\lambda\\
&\leq& \frac{1}{\sigma(\ell,w)}\int_{-\ell}^\ell \Bigg\|\xi(\lambda)\Bigg[h(\lambda+\tau,\zeta(\lambda+\tau))-ah\left(\lambda,\frac{1}{a}\zeta(\lambda+\tau)\right)\Bigg]\Bigg\|w(\lambda)d\lambda\\
&&+\frac{1}{\sigma(\ell,w)}\int_{-\ell}^\ell \Bigg\|\xi(\lambda)\Bigg[ah\left(\lambda,\frac{1}{a}\zeta(\lambda+\tau)\right)-ah(\lambda,\zeta(\lambda))\Bigg]\Bigg\|w(\lambda)d\lambda\\
&:=&\mathbb{J}_1+\mathbb{J}_2.
\end{eqnarray*}
Clearly, $\displaystyle{\lim_{\ell\rightarrow\infty}\mathbb{J}_1=0}$. And we can see that if $$\left\|\xi(\lambda)[\zeta(\lambda+\tau)-a\zeta(\lambda)]\right\|_{\mathbb{F}}<|a|\mathbbm{d},$$ then
$$\left\|\frac{\xi(\lambda)}{a}\zeta(\lambda+\tau)-\xi(\lambda)\zeta(\lambda)\right\|_{\mathbb{F}}=\left\|\frac{1}{a}\xi(\lambda)[\zeta(\lambda+\tau)-a\zeta(\lambda)]\right\|_{\mathbb{F}}<\mathbbm{d}.$$
By $(I_2^{\prime\prime})$, we can get
\begin{eqnarray}
&&\frac{1}{\sigma(\ell,w)}\int_{-\ell}^\ell \Bigg\|\xi(\lambda)\Bigg[ah\left(\lambda,\frac{1}{a}\zeta(\lambda+\tau)\right)-ah(\lambda,\zeta(\lambda))\Bigg]\Bigg\|w(\lambda)d\lambda\nonumber\\
&=&\frac{|a|}{\sigma(\ell,w)}\int_{-\ell}^\ell\left\|\xi(\lambda)\left[h\left(\lambda,\frac{1}{\xi(\lambda)}\frac{\xi(\lambda)}{a}\zeta(\lambda+\tau)\right)-h\left(\lambda,\frac{1}{\xi(\lambda)}\xi(\lambda)\zeta(\lambda)\right)\right]
\right\| w(\lambda)d\lambda\nonumber\\
&=&\frac{|a|}{\sigma(\ell,w)}\int_{[-\ell,\ell]\backslash\Sigma_{\ell,|a|\mathbbm{d}}(\zeta)}\left\|\xi(\lambda)\left[h\left(\lambda,\frac{1}{\xi(\lambda)}\frac{\xi(\lambda)}{a}\zeta(\lambda+\tau)\right)-h\left(\lambda,\frac{1}{\xi(\lambda)}\xi(\lambda)\zeta(\lambda)\right)\right]\right\|w(\lambda)d\lambda\nonumber\\
&&+\frac{|a|}{\sigma(\ell,w)}\int_{\Sigma_{\ell,|a|\mathbbm{d}}(\zeta)}\left\|\xi(\lambda)\left[h\left(\lambda,\frac{1}{\xi(\lambda)}\frac{\xi(\lambda)}{a}\zeta(\lambda+\tau)\right)-h\left(\lambda,\frac{1}{\xi(\lambda)}\xi(\lambda)\zeta(\lambda)\right)\right]\right\|w(\lambda)d\lambda\nonumber\\
&\leq&\frac{1}{\sigma(\ell,w)}\int_{-\ell}^\ell G(\lambda)w(\lambda)d\lambda\cdot|a|\mu+2|a|C\frac{1}{\sigma(\ell,w)}\int_{\Sigma_{\ell,|a|\mathbbm{d}}(\zeta)}w(\lambda)d\lambda.\nonumber
\end{eqnarray}
By Lemma \ref{jdkkd}, the conditions $\zeta\in {\rm{PSAP}}_{\xi,\tau,a}(\mathbb{R},\mathbb{F},w)$, $G(\cdot)\in \mathcal{B}_w$ and the arbitrariness of $\mu$, we arrive at
$\displaystyle{\lim_{\ell\rightarrow\infty}\mathbb{J}_2=0}$. It follows that
\begin{center}
$\displaystyle{\frac{1}{\sigma(\ell,w)}\int_{-\ell}^\ell \left\|\xi(\lambda)[\overline{\Psi}_h(\zeta)(\lambda+\tau)-a\overline{\Psi}_h(\zeta)(\lambda)]\right\|w(\lambda)d\lambda\rightarrow0},\ {\rm{as}}\ \ell\rightarrow \infty,$
\end{center}
which shows that $\overline{\Psi}_h(\zeta)\in {\rm{PSAP}}_{\xi,\tau,a}(\mathbb{R},\mathbb{E},w)$. Hence, the theorem is established.
\end{proof}

From Theorem \ref{dndk}, we can derive the following two corollaries.
\begin{Corollary}\label{dhjkdkd}
Let $h\in C(\mathbb{R}\times \mathbb{F},\mathbb{E})$ satisfy $(I_1)$ and the following condition:\\
$(J)$ There exists a function $G(\cdot)\in \mathcal{B}_w$ such that
for any $y_i\in \mathbb{F}\ (i=1,2)$,
$$\|\mathcal{H}(\lambda,y_1)-\mathcal{H}(\lambda,y_2)\|\leq G(\lambda)\|y_1-y_2\|_{\mathbb{F}},\  \lambda\in \mathbb{R},$$
where $\mathcal{H}$ is defined by formula $(\ref{bdjkdb})$.
 Then for any $\zeta\in {\rm{PSAP}}_{\xi,\tau,a}(\mathbb{R},\mathbb{F},w)$, we have $\overline{\Psi}_h(\zeta)\in {\rm{PSAP}}_{\xi,\tau,a}(\mathbb{R},\mathbb{E},w)$.
\end{Corollary}
\begin{proof}
We can easily verify that condition $(J)$ implies condition $(I_2^{\prime\prime})$, so by Theorem \ref{dndk} we draw the conclusion of this corollary.
\end{proof}

\begin{Corollary}\label{vnwkejj}
Let $h\in C(\mathbb{R}\times \mathbb{F},\mathbb{E})$ satisfy $(I_1)$ and the following condition:\\
$(J^\prime)$ There exists a function $G(\cdot)\in L^q(\mathbb{R},[0,\infty),w), \ q\in[1,\infty)$ such that
for any $y_i\in \mathbb{F}\ (i=1,2)$,
$$\|\mathcal{H}(\lambda,y_1)-\mathcal{H}(\lambda,y_2)\|\leq G(\lambda)\|y_1-y_2\|_{\mathbb{F}},\ \lambda\in \mathbb{R},$$
where $\mathcal{H}$ is defined by formula $(\ref{bdjkdb})$.
 Then for any $\zeta\in {\rm{PSAP}}_{\xi,\tau,a}(\mathbb{R},\mathbb{F},w)$, $\overline{\Psi}_h(\zeta)\in {\rm{PSAP}}_{\xi,\tau,a}(\mathbb{R},\mathbb{E},w)$.
\end{Corollary}
\begin{proof}
We only need to prove that $L^q(\mathbb{R},[0,\infty),w)\subseteq \mathcal{B}_w,$ for any $q\in[1,\infty)$. It is obvious that $L^1(\mathbb{R},[0,\infty),w)\subseteq \mathcal{B}_w$. When $q>1$, set $\displaystyle{p=\frac{q-1}{q}}$. By generalized H${\rm{\ddot{o}}}$lder inequality, we get
$$\frac{1}{\sigma(\ell,w)}\int_{-\ell}^\ell h(\lambda)w(\lambda)d\lambda\leq \frac{1}{\left(\sigma(\ell,w)\right)^{1-1/p}}\left(\int_{-\ell}^\ell h^q(\lambda)w(\lambda)d\lambda\right)^{1/q},$$
which implies that $L^q(\mathbb{R},[0,\infty),w)\subseteq \mathcal{B}_w$ for any $q\in[1,\infty)$. So the desired result follows from an application of Corollary \ref{dhjkdkd}.
\end{proof}
\begin{Remark}
 The condition imposed on $\xi$ in compositions results from Theorem \ref{djdk} to Corollary \ref{vnwkejj} is very weak. We do not require $\xi(\cdot)\in \widetilde{\mathbb{U}}$, nor do we impose any specific relationship between $\xi(\cdot)$ and $\tau,\ a$, which substantially broadens the applicable range of such weighted pseudo $\emph{S}$-asymptotically $(\xi,\tau,a)$-periodic functions.
\end{Remark}
To study the equations (\ref{equation2}) and (\ref{equation1}), we can extend the definitions and theorems from previous parts of this section  from $\mathbb{R}$ to $[0,\infty)$ and just provide a few such definitions and theorems below and omit the rest.

For given $0\neq a\in \mathbb{C},\ \tau\in (0,\infty)$, define the following set
$$BC_{\xi}([0,\infty),\mathbb{E})=\left\{ h\in C([0,\infty),\mathbb{E}):\ \sup_{\lambda\geq 0}\|\xi(\lambda)h(\lambda)\|<\infty\right\}.$$
\begin{Lemma}
Let $h\in C([0,\infty),\mathbb{E})$. Then $h\in BC_{\xi}([0,\infty),\mathbb{E})$ if and only if $h(\lambda)=\frac{1}{\xi(\lambda)}h_1(\lambda)$, $ h_1(\lambda)\in BC([0,\infty),\mathbb{E}).$
\end{Lemma}
\begin{Lemma}
The set $BC_{\xi}([0,\infty),\mathbb{E})$ is a Banach space with the norm
$$\|h\|_{b\xi}=\displaystyle{\sup_{\lambda\geq 0}}\|\xi(\lambda)h(\lambda)\|.$$
\end{Lemma}

\begin{Definition}
For a function  $h$ such that $h(\cdot),\ h(\cdot+\tau)\in BC_{\xi}([0,\infty),\mathbb{E})$, it is called $\emph{S}$-asymptotically $(\xi,\tau,a)$-periodic if
$$\displaystyle{\lim_{\lambda\rightarrow +\infty}}\left\|\xi(\lambda)[h(\lambda+\tau)-ah(\lambda)]\right\|=0.$$
\end{Definition}
The set containing such functions is denoted by ${\rm{SAP}}_{\xi,\tau,a}([0,\infty),\mathbb{E})$.

Set
$${\rm{SAP}}_{\tau,a,0}([0,\infty),\mathbb{E})=\{h\in {\rm{SAP}}_{\tau,a}([0,\infty),\mathbb{E}):\ h(0)=0\},$$
and
$$BC_{\tau,a,0}(\mathbb{E})=\{h\in BC_{\tau,a}([-\kappa,\infty),\mathbb{E}):\ h|_{\lambda\geq 0}\in BC_{\tau,a}([0,\infty),\mathbb{E}),\ h|_{[-\kappa,0]}=0\}.$$

\begin{Definition}\label{hddj}
For a function  $h$ such that $h(\cdot),\ h(\cdot+\tau)\in BC_{\xi}([0,\infty),\mathbb{E})$, it is called weighted pseudo $\emph{S}$-asymptotically $(\xi,\tau,a)$-periodic with respect to $w$, or $w$-pseudo $\emph{S}$-asymptotically $(\xi,\tau,a)$-periodic, if it satisfies
$$\lim_{\ell\rightarrow \infty}\frac{1}{\sigma(\ell,w)}\int_0^\ell \left\|\xi(\lambda)[h(\lambda+\tau)-ah(\lambda)]\right\|w(\lambda)d\lambda=0,$$
where $\sigma(\ell,w)=\displaystyle{\int_{0}^\ell w(\theta)d\theta}$.
\end{Definition}
${\rm{PSAP}}_{\xi,\tau,a}([0,\infty),\mathbb{E},w)$ denotes the set of all $w$-pseudo $\emph{S}$-asymptotically $(\xi,\tau,a)$-periodic functions from $[0,\infty)$ to $\mathbb{E}$.

\begin{Theorem}\label{dhjd}
Suppose $\{\mathcal{T}(\lambda)\}_{\lambda\geq 0}$ is a family of bounded operators on $\mathbb{E}$ satisfying that $\{\xi(\lambda)\mathcal{T}(\lambda)\}_{\lambda\geq 0}$ is strongly continuous and $\xi(\lambda)\mathcal{T}(\lambda)\in
L^1([0,\infty),\mathcal{B}(\mathbb{E}))$.  Also suppose  $\xi\in \widetilde{\mathbb{U}}$, $w_r\sim w$ and $|\xi_r|\sim |\xi|$ for any $r>0$. Then for any $h\in {\rm{PSAP}}_{\xi,\tau,a}([0,\infty),\mathbb{E},w)$, the convolution
$$\phi(\lambda):=\int_0^\lambda \mathcal{T}(\lambda-\theta)h(\theta)d\theta$$
also belongs to ${\rm{PSAP}}_{\xi,\tau,a}([0,\infty),\mathbb{E},w)$.
\end{Theorem}
\begin{Theorem}\label{djdk11}
Let $h\in C([0,\infty)\times \mathbb{F},\mathbb{E})$ satisfy the following conditions:\\
$(I_1)$ $(i)$ For any $y\in \mathbb{F}$, $h(\cdot,y)\in BC_{\xi}([0,\infty),\mathbb{E})$, and $\displaystyle{\sup_{ y\in \mathbb{F}}\left\|h(\cdot,y)\right\|_{b\xi}<\infty}$;\\
$(ii)$ Given the limit $\displaystyle{\lim_{\ell\rightarrow \infty}\frac{1}{\sigma(\ell,w)}\int_{0}^\ell \left\|\xi(\lambda)[h(\lambda+\tau,ay)-ah(\lambda,y)]\right\|w(\lambda)d\lambda}$, it converges to zero uniformly with respect to $y\in \mathbb{F}$.\\
$(I_2)$ There exists a positive constant $K$ such that
$$\|h(\lambda,y_1)-h(\lambda,y_2)\|\leq K\|y_1-y_2\|_{\mathbb{F}},\ \forall   y_i\in \mathbb{F}\ (i=1,2),\ \lambda\geq 0.$$
Then for any $\zeta\in {\rm{PSAP}}_{\xi,\tau,a}([0,\infty),\mathbb{F},w)$, we have $\overline{\Psi}_h(\zeta)\in {\rm{PSAP}}_{\xi,\tau,a}([0,\infty),\mathbb{E},w)$.
\end{Theorem}

\section{Weighted pseudo $\emph{S}$-asymptotically $(\xi,\tau,a)$-periodic solutions for Syst. (\ref{equation2})}

By applying the properties of weighted pseudo $\emph{S}$-asymptotically $(\xi,\tau,a)$-periodic functions introduced in Section 3,
we proceed to analyze the existence and uniqueness of such pseudo periodic mild solutions to the fractional evolution system (\ref{equation2}).

\begin{Theorem}\label{hdjkhd}
Suppose the family $\{\xi(\lambda)\lambda^{\tilde{\iota}-1}\Psi_2(\lambda)\}_{\lambda> 0}$ is strongly continuous, and that there exists a function $\mathfrak{R}(\lambda)\in L^1([0,\infty),[0,\infty))$ such that
 $\left\|\xi(\lambda)\lambda^{\tilde{\iota}-1}\Psi_2(\lambda)\right\|\leq \mathfrak{R}(\lambda)$ for almost every $\lambda> 0$. Further suppose  $\xi\in \widetilde{\mathbb{U}}$, $w_r\sim w$ and $|\xi_r|\sim |\xi|$ for all $r> 0$. If
$\varrho\in C([0,\infty)\times \mathbb{E},\mathbb{E})$ satisfies condition $(I_1)$ and the Lipschitz condition
$$\|\varrho(\lambda,\iota_1)-\varrho(\lambda,\iota_2)\|\leq K\|\iota_1-\iota_2\|,\ \forall  \iota_i\in \mathbb{E}\ (i=1,2),\ {\rm{a.e.}}\ \lambda\geq 0,$$
with the positive Lipschitz constant $K$ satisfying
 $$K<\frac{1}{\mathcal{M}\|\mathfrak{R}\|_{L^1((0,\infty),[0,\infty))}},$$ then Syst. $(\ref{equation2})$ admits a unique $w$-pseudo $S$-asymptotically $(g,\tau,a)$-periodic mild solution.
\end{Theorem}
\begin{proof}
Drawing on Eq. $(\ref{dnkfhfkhis1})$ along with Theorems \ref{dhjd} and \ref{djdk11}, we construct an operator $\mathcal{Q}$ on the Banach space $\left({\rm{PSAP}}_{\xi,\tau,a}([0,\infty),\mathbb{E},w),\|\cdot\|_{b\xi}\right)$,
which takes the form:
\begin{equation}\label{ddjkdj}
(\mathcal{Q}\nu)(\lambda)=\displaystyle{\Psi_1(\lambda)\nu_0+\int_0^\lambda(\lambda-\theta)^{\tilde{\iota}-1}\Psi_2(\lambda-\theta)\varrho(\theta,\nu(\theta))d\theta,\ \lambda\geq 0.}
\end{equation}
Now take any two functions $\nu,\ \tilde{\nu}\in {\rm{PSAP}}_{\xi,\tau,a}([0,\infty),\mathbb{E},w)$. Using the property $\xi\in \widetilde{\mathbb{U}}$, we derive the following norm bound:
\begin{eqnarray*}
&&\left\|\xi(\lambda)[\mathcal{Q}(\nu)(\lambda)-\mathcal{Q}(\tilde{\nu})(\lambda)]\right\|\\
&\leq&\int_0^\lambda\left\|\xi(\lambda)(\lambda-\theta)^{\tilde{\iota}-1}\Psi_2(\lambda-\theta)[\varrho(\theta,\nu(\theta))-\varrho(\theta,\tilde{\nu}(\theta))]\right\|d\theta\\
&\leq&\int_0^\lambda K\mathcal{M}\left\|\xi(\lambda-\theta)(\lambda-\theta)^{\tilde{\iota}-1}\Psi_2(\lambda-\theta)\xi(\theta)[\nu(\theta)-\tilde{\nu}(\theta)]\right\|d\theta\\
&\leq&K\mathcal{M}\|\nu-\tilde{\nu}\|_{b\xi}\int_0^\lambda \mathfrak{R}(\lambda-\theta)d\theta\\
&=&K\mathcal{M}\|\mathfrak{R}\|_{L^1((0,\infty),[0,\infty))}\|\nu-\tilde{\nu}\|_{b\xi}.
\end{eqnarray*}
It follows immediately that
$$\|\mathcal{Q}\nu-\mathcal{Q}\tilde{\nu}\|_{b\xi}\leq K\mathcal{M}\|\mathfrak{R}\|_{L^1((0,\infty),[0,\infty))}\|\nu-\tilde{\nu}\|_{b\xi},$$
which confirms that $\mathcal{Q}$ is a contraction mapping. The proof of the theorem is then completed by  using the Banach contraction principle.
\end{proof}
\begin{Remark}\label{hddjlak}
Let the semigroup $\{\Psi(\lambda)\}_{\lambda\geq 0}$ be uniformly bounded, and let $\xi(\lambda)=a^\uparrow(-\lambda)$ with $|a|>1$.  One can readily verify the existence of  $\mathfrak{R}(\lambda)$ (as it appears in Theorem \ref{hdjkhd}) directly from (\ref{dnkhfkhis3}).
\end{Remark}

\begin{Theorem}
Assume the family $\{\xi(\lambda)\lambda^{\tilde{\iota}-1}\Psi_2(\lambda)\}_{\lambda> 0}$ is strongly continuous, and that there exists a nonnegative  nonincreasing function $\mathfrak{R}(\lambda)$ defined
on $[0,\infty)$ such that
 $$\left\|\xi(\lambda)\lambda^{\tilde{\iota}-1}\Psi_2(\lambda)\right\|\leq \mathfrak{R}(\lambda)$$ holds for almost every $\lambda>0$, with the infinite series satisfying  $$\displaystyle{\sum_{n=0}^\infty \mathfrak{R}(n)<\infty}.$$ Also assume  $w_r\sim w$, $|\xi_r|\sim |\xi|$ for all $r\in \mathbb{R}$ and $\xi\in \widetilde{\mathbb{U}}$. Suppose
$\varrho\in C([0,\infty)\times \mathbb{E},\mathbb{E})$ satisfies condition $(I_1)$ and there exists a nonnegative, locally integrable function  $K(\lambda)$ on $[0,\infty)$ such that
\begin{equation}\label{bdcmsb}
\left\|\xi(\lambda)\left[\varrho\left(\lambda,\frac{1}{\xi(\lambda)}\iota_1\right)-\varrho\left(\lambda,\frac{1}{\xi(\lambda)}\iota_2\right)\right]\right\|\leq K(\lambda)\|\iota_1-\iota_2\|,\ \forall \iota_i\in \mathbb{E}\ (i=1,2),\ {\rm{a.e.}}\ \lambda\geq 0,
\end{equation}
with $\displaystyle{\sup_{\lambda\geq 0}\int_\lambda^{\lambda+1}K(\theta)d\theta}<\infty$. Let $\mathcal{R}=\displaystyle{\sum_{n=0}^\infty \mathfrak{R}(n)}$ and $K_1
=\displaystyle{\sup_{\lambda\geq 0}\int_\lambda^{\lambda+1}K(\theta)d\theta}$. If $\mathcal{M}\mathcal{R}K_1<1$, then there exists a unique $w$-pseudo $S$-asymptotically $(\xi,\tau,a)$-periodic mild solution for Syst. $(\ref{equation2})$.
\end{Theorem}
\begin{proof}
By the fact $\displaystyle{\sum_{n=0}^\infty \mathfrak{R}(n)<\infty}$ and the noninreasing property of $\mathfrak{R}(\lambda)$, one can easily see $\mathfrak{R}\in L^1([0,\infty),[0,\infty))$. By Theorem \ref{dhjd}
and Corollary \ref{vnwkejj}, we consider the operator $\mathcal{Q}$ defined on the Banach space $\left({\rm{PSAP}}_{\xi,\tau,a}([0,\infty),\mathbb{E},w),\|\cdot\|_{b\xi}\right)$ by $(\ref{ddjkdj})$.
For any $\nu,\ \tilde{\nu}\in {\rm{PSAP}}_{\xi,\tau,a}([0,\infty),\mathbb{E},w)$, we have
\begin{eqnarray*}
&&\left\|\xi(\lambda)[\mathcal{Q}(\nu)(\lambda)-\mathcal{Q}(\tilde{\nu})(\lambda)]\right\|\\
&\leq&\mathcal{M}\int_0^\lambda\left\|\xi(\lambda-\theta)(\lambda-\theta)^{\tilde{\iota}-1}\Psi_2(\lambda-\theta)\xi(\theta)[\varrho(\theta,\nu(\theta))-\varrho(\theta,\tilde{\nu}(\theta))]\right\|d\theta\\
&\leq&\mathcal{M}\int_0^\lambda K(\theta)\left\|\xi(\lambda-\theta)(\lambda-\theta)^{\tilde{\iota}-1}\Psi_2(\lambda-\theta)\right\|\left\|\xi(\theta)[\nu(\theta)-\tilde{\nu}(\theta)]\right\|d\theta\\
&\leq&\mathcal{M}\left(\sum_{n=0}^\infty \int_{\lambda-(n+1)}^{\lambda-n}K(\theta)\mathfrak{R}(\lambda-\theta)d\theta\right)\|\nu-\tilde{\nu}\|_{b\xi}\\
&\leq&\mathcal{M}\left(\sum_{n=0}^\infty \mathfrak{R}(n)\int_{\lambda-(n+1)}^{\lambda-n}K(\theta)d\theta\right)\|\nu-\tilde{\nu}\|_{b\xi}\\
&\leq&\mathcal{M}\mathcal{R}K_1\|\nu-\tilde{\nu}\|_{b\xi},
\end{eqnarray*}
meaning $\|\mathcal{Q}\nu-\mathcal{Q}\tilde{\nu}\|_{b\xi}\leq \mathcal{M}\mathcal{R}K_1\|\nu-\tilde{\nu}\|_{b\xi}.$ Since $\mathcal{M}\mathcal{R}K_1<1$,  $\mathcal{Q}$ is a contraction mapping, and hence the conclusion follows from the contraction principle.
\end{proof}

\begin{Theorem}\label{djddk}
Suppose $\{\xi(\lambda)\lambda^{\tilde{\iota}-1}\Psi_2(\lambda)\}_{\lambda\geq 0}$ is strongly continuous,
 $$\xi(\lambda)\lambda^{\tilde{\iota}-1}\Psi_2(\lambda)\in L^1([0,\infty),\mathbb{E}),$$ and  there exists a positive constant $\mathfrak{L}$ such that $$\left\|\xi(\lambda)\lambda^{\tilde{\iota}-1}\Psi_2(\lambda)\right\|\leq \mathfrak{L}$$ for ${\rm{a.e.}}$ $\lambda>0$. Also suppose  $w_r\sim w$, $|\xi_r|\sim |\xi|$ for all $r\in \mathbb{R}$ and $\xi\in \widetilde{\mathbb{U}}$. If
$\varrho\in C([0,\infty)\times \mathbb{E},\mathbb{E})$ satisfies $(I_1)$ and $(\ref{bdcmsb})$ with $K(\lambda)\in L^1([0,\infty),[0,\infty))$, then there exists a unique $w$-pseudo $S$-asymptotically $(\xi,\tau,a)$-periodic mild solution for Syst. $(\ref{equation2})$.
\end{Theorem}
\begin{proof}
Consider again the operator $\mathcal{Q}:\ {\rm{PSAP}}_{\xi,\tau,a}([0,\infty),\mathbb{E},w)\rightarrow {\rm{PSAP}}_{\xi,\tau,a}([0,\infty),\mathbb{E},w)$ defined by $(\ref{ddjkdj})$. The well-definedness of the operator $\mathcal{Q}$ can be guaranteed by Theorem \ref{dhjd}
and Corollary \ref{vnwkejj}. For any $\nu,\ \tilde{\nu}\in {\rm{PSAP}}_{\xi,\tau,a}([0,\infty),\mathbb{E},w)$, we have
\begin{eqnarray*}
&&\left\|\xi(\lambda)[\mathcal{Q}(\nu)(\lambda)-\mathcal{Q}(\tilde{\nu})(\lambda)]\right\|\\
&\leq&\mathcal{M}\int_0^\lambda\left\|\xi(\lambda-\theta)(\lambda-\theta)^{\tilde{\iota}-1}\Psi_2(\lambda-\theta)K(\theta)\xi(\theta)[\nu(\theta)-\tilde{\nu}(\theta)]\right\|d\theta\\
&\leq&\mathcal{M}\mathfrak{L}\int_0^\lambda K(\theta)d\theta\|\nu-\tilde{\nu}\|_{b\xi}.
\end{eqnarray*}
By induction, we can obtain for any $n\in \mathbb{N}_+$,
\begin{eqnarray*}
&&\left\|\xi(\lambda)[(\mathcal{Q}^n\nu)(\lambda)-(\mathcal{Q}^n\tilde{\nu})(\lambda)]\right\|\\
&\leq& \frac{(\mathcal{M}\mathfrak{L})^n}{(n-1)!}\left(\int_0^\lambda K(\theta)\left(\int_0^\theta K(\sigma)d\sigma\right)^{n-1}d\theta\right)\|\nu-\tilde{\nu}\|_{b\xi}\\
&\leq& \frac{1}{n!}\left(\mathcal{M}\mathfrak{L}\|K\|_{L^1([0,\infty),[0,\infty))}\right)^n\|\nu-\tilde{\nu}\|_{b\xi}.
\end{eqnarray*}
So,
$$\|\mathcal{Q}^n\nu-\mathcal{Q}^n\tilde{\nu}\|_{b\xi}\leq \frac{1}{n!}\left(\mathcal{M}\mathfrak{L}\|K\|_{L^1([0,\infty),[0,\infty))}\right)^n\|\nu-\tilde{\nu}\|_{b\xi}.$$
For $n$ large enough, the fact $\frac{1}{n!}\left(\mathcal{M}\mathfrak{L}\|K\|_{L^1([0,\infty),[0,\infty))}\right)^n<1$ ensures that $\mathcal{Q}$ has a unique fixed point
 $\nu\in {\rm{PSAP}}_{\xi,\tau,a}([0,\infty),\mathbb{E},w)$. We have thus finished the proof of this theorem.
\end{proof}

\begin{Theorem}
Suppose $\{\xi(\lambda)\lambda^{\tilde{\iota}-1}\Psi_2(\lambda)\}_{\lambda\geq 0}$ is strongly continuous and uniformly exponential stable. Also suppose  $w_r\sim w$, $|\xi_r|\sim |\xi|$ for all $r\in \mathbb{R}$ and $\xi\in \widetilde{\mathbb{U}}$. If
$\varrho\in C([0,\infty)\times \mathbb{E},\mathbb{E})$ satisfies $(I_1)$ and $(\ref{bdcmsb})$
with $K(\lambda)\in L^q([0,\infty),[0,\infty)),\ q\in[1,\infty)$, then there exists a unique $w$-pseudo $S$-asymptotically $(\xi,\tau,a)$-periodic mild solution for Syst. $(\ref{equation2})$.
\end{Theorem}
\begin{proof}
Similar to previous theorems, we consider the operator $$\mathcal{Q}:\ {\rm{PSAP}}_{\xi,\tau,a}([0,\infty),\mathbb{E},w)\rightarrow {\rm{PSAP}}_{\xi,\tau,a}([0,\infty),\mathbb{E},w)$$ defined by
$(\ref{ddjkdj})$. The family $\{\xi(\lambda)\lambda^{\tilde{\iota}-1}\Psi_2(\lambda)\}_{\lambda>0}$ exhibits uniform exponential stability, hence, one can find positive constants $\verb"R"$
and $\varpi$ such that $$\|\xi(\lambda)\lambda^{\tilde{\iota}-1}\Psi_2(\lambda)\|\leq \verb"R"{\rm{e}}^{-\varpi \lambda} \quad \mbox{for any} \ \ \lambda>0.$$ The integrability of the family
 $\{\xi(\lambda)\lambda^{\tilde{\iota}-1}\Psi_2(\lambda)\}_{\lambda> 0}$ on $(0,\infty)$ follows as a consequence. The well-definedness of the operator $\mathcal{Q}$ can be inferred from Theorem \ref{dhjd} and Corollary \ref{vnwkejj}.

The case $q=1$ immediately yields $\|\xi(\lambda)\lambda^{\tilde{\iota}-1}\mathbb{W}(\lambda)\|\leq \verb"R"$, enabling the proof to be carried out similarly to Theorem \ref{djddk}.

In order to consider the case $1<q<\infty$, let us introduce a modified equivalent norm $\|\cdot\|_{b\xi;\omega}$ on ${\rm{PSAP}}_{\xi,\tau,a}([0,\infty),\mathbb{E},w)$ defined as follows by
$$\|\nu\|_{b\xi;\omega}=\sup_{\lambda\geq 0}\left\|{\rm{e}}^{-\beta \omega(\lambda)}\xi(\lambda)\nu(\lambda)\right\|,\ \nu\in {\rm{PSAP}}_{\xi,\tau,a}([0,\infty),\mathbb{E},w),$$
where $\beta$ is a fixed large positive constant and $\omega(\lambda):=\displaystyle{\int_0^\lambda K^q(\theta)d\theta}$. For any $\nu,\ \tilde{\nu}\in {\rm{PSAP}}_{\xi,\tau,a}([0,\infty),\mathbb{E},w)$, we get
\begin{eqnarray*}
&&\left\|\xi(\lambda)[\mathcal{Q}(\nu)(\lambda)-\mathcal{Q}(\tilde{\nu})(\lambda)]\right\|\\
&\leq&\mathcal{M}\int_0^\lambda\left\|\xi(\lambda-\theta)(\lambda-\theta)^{\tilde{\iota}-1}\Psi_2(\lambda-\theta)g(\theta)[\varrho(\theta,\nu(\theta))-\varrho(\theta,\tilde{\nu}(\theta))]\right\|d\theta\\
&\leq&\mathcal{M}\verb"R"\int_0^\lambda {\rm{e}}^{-\varpi(\lambda-\theta)}K(\theta)\left\|\xi(\theta)[\nu(\theta)-\tilde{\nu}(\theta)]\right\|d\theta\\
&\leq&\mathcal{M}\verb"R"\int_0^\lambda {\rm{e}}^{-\varpi(\lambda-\theta)}K(\theta){\rm{e}}^{\beta \omega(\theta)} ds\left\|\nu-\tilde{\nu}\right\|_{b\xi;\omega}\\
&\leq&\mathcal{M}\verb"R"\left[\int_0^\lambda {\rm{e}}^{q\beta \omega(\theta)}K^p(\theta)d\theta\right]^{1/q}\left[\int_0^\lambda {\rm{e}}^{-\varpi \frac{q(\lambda-\theta)}{q-1}}d\theta\right]^{(q-1)/q}\left\|\nu-\tilde{\nu}\right\|_{b\xi;\omega}\\
&\leq&\mathcal{M}\verb"R"\left(\frac{q\varpi}{q-1}\right)^{(1-q)/q}\left[\int_0^\lambda {\rm{e}}^{q\beta \omega(\theta)}d\omega(\theta)\right]^{1/q}\left\|\nu-\tilde{\nu}\right\|_{b\xi;\omega}\\
&\leq&\mathcal{M}\verb"R"\left(\frac{q\varpi}{q-1}\right)^{(1-q)/q}(q\beta)^{-\frac{1}{q}}\left({\rm{e}}^{q\beta \omega(\lambda)}-1\right)^{1/q}\left\|\nu-\tilde{\nu}\right\|_{b\xi;\omega}.
\end{eqnarray*}
Hence,
$$\left\|\mathcal{Q}\nu-\mathcal{Q}\tilde{\nu}\right\|_{b\xi;\omega}\leq \mathcal{M}\verb"R"\left(\frac{q\varpi}{q-1}\right)^{(1-q)/q}(q\beta)^{-1/q}\left\|\nu-\tilde{\nu}\right\|_{b\xi;\omega}.$$
For sufficiently large $\beta$, the operator $\mathcal{Q}$ is contractive. Consequently, it admits a unique fixed point $\nu\in {\rm{PSAP}}_{\xi,\tau,a}([0,\infty),\mathbb{E},w)$ that serves as a mild solution to Syst. $(\ref{equation2})$. The proof of this theorem is thus finished.
\end{proof}

\section{$\emph{S}$-asymptotically $(\tau,a)$-periodic solutions for Syst. (\ref{equation1})}
Research on $\emph{S}$-asymptotically $(\tau,a)$-periodic solutions, as a special case of weighted pseudo $\emph{S}$-asymptotically $(\xi,\tau,a)$-periodic solutions,
for fractional evolution equations seems to have been unexplored. In this section, we will present an existence and uniqueness theorem for $\emph{S}$-asymptotically
$(\tau,a)$-periodic mild solutions to Sys. (\ref{equation1})  which extends the results in \cite{LiL1}. Before giving the main theorem, we first pose some assumptions
 for the semigroup $\{\Psi(\lambda)\}_{t\geq 0}$ and the nonlinear term $\varrho:\ [0,\infty)\times C([-\kappa,0],\mathbb{E})\rightarrow \mathbb{E}$ below:\\
$(J_1)$ The semigroup $\{\Psi(\lambda)\}_{\lambda\geq 0}$ is of type $\mathbf{T}(M,\alpha)$.\\
$(J_2)$ $\displaystyle{\sup_{\lambda\geq 0}|a|^{\uparrow}(-\lambda)\int_0^\lambda \frac{\|\varrho(\theta,0)\|}{(\lambda-\theta)^{1-\tilde{\iota}}}d\theta<\infty.}$\\
$(J_3)$ There exists $\tau>0$ such that  $\displaystyle{\lim_{\lambda\rightarrow \infty}\|\varrho(\lambda+\tau,\vartheta)-\varrho(\lambda,\vartheta)\|=0}$ for any $\vartheta\in C([-\kappa,0],\mathbb{E})$.\\
$(J_4)$ There exists a locally integrable function $L(\cdot):\ [0,\infty)\rightarrow \mathbb{E}$ such that for any $\vartheta_i\in C([-\kappa,0],\mathbb{E})\ (i=1,2)$,
\begin{center}
$\|\varrho(\lambda,\vartheta_1)-\varrho(\lambda,\vartheta_2)\|\leq L(\lambda)\|\vartheta_1-\vartheta_2\|_{[-\kappa,0]}$, $\lambda\geq 0$,
\end{center}
 and
\begin{equation}\label{nkfn}
\displaystyle{\sup_{\lambda\geq 0}|a|^{\uparrow}(-\lambda)\int_0^\lambda \frac{|a|^{\uparrow}(\theta)L(\theta)}{(\lambda-\theta)^{1-\tilde{\iota}}}d\theta}<\frac{\Gamma(\tilde{\iota})}{M}.
\end{equation}
$(J_5)$ There exists a locally integrable function $N(\cdot):\ [0,\infty)\rightarrow \mathbb{E}$  such that for any $\vartheta_i\in C([-\kappa,0],\mathbb{E})\ (i=1,2)$,
\begin{center}
$\|\varrho(\lambda,\vartheta_1)-a\varrho(\lambda,\vartheta_2)\|\leq N(\lambda)\|\vartheta_1-a\vartheta_2\|_{[-\kappa,0]}$, \ $ \lambda\geq 0$,
\end{center}
and
\begin{equation}\label{nkfn}
\displaystyle{\sup_{\lambda\geq 0}|a|^{\uparrow}(-\lambda)\int_0^\lambda \frac{|a|^{\uparrow}(\theta)N(\theta)}{(\lambda-\theta)^{1-\tilde{\iota}}}d\theta}<\infty.
\end{equation}

\begin{Theorem}\label{cnknn}
If the assumptions $(J_1)-(J_5)$ hold and $|a|\geq 1$, then there exists a unique $\emph{S}$-asymptotically $(\tau,a)$-periodic mild solution for Syst. (\ref{equation1}).
\end{Theorem}
\begin{proof}
By $(J_2),(J_4)$ and $(J_5)$, we can set
\[
\Phi_1=\displaystyle{\sup_{\lambda\geq 0}|a|^{\uparrow}(-\lambda)\int_0^\lambda \frac{|a|^{\uparrow}(\theta)L(\theta)}{(\lambda-\theta)^{1-\tilde{\iota}}}d\theta},
\]
\[
\Phi_2=\displaystyle{\sup_{\lambda\geq 0}|a|^{\uparrow}(-\lambda)\int_0^\lambda \frac{|a|^{\uparrow}(\theta)N(\theta)}{(\lambda-\theta)^{1-\tilde{\iota}}}d\theta},
\]
and
$$\displaystyle{\Lambda=\sup_{\lambda\geq 0}|a|^{\uparrow}(-\lambda)\int_0^\lambda \frac{\|\varrho(\theta,0)\|}{(\lambda-\theta)^{1-\tilde{\iota}}}d\theta.}$$
For $\eta\in C([-\kappa,0],\mathbb{E})$ given in Eq. (\ref{equation1}), the function $z(\lambda)$ is defined as
\[
z(\lambda)=
\left\{
\begin{array}{lr}
\eta(\lambda),\ \qquad \lambda\in [-\kappa,0],\\
\Psi_1(\lambda)\eta(0),\ \lambda\geq 0.
\end{array}
\right.
\]
Then $z\in C([-\kappa,\infty),\mathbb{E})$. Let
$$\nu(\lambda)=\varsigma(\lambda)+z(\lambda),\ \lambda\in[-\kappa,\infty).$$
One can easily verify that $\nu$ fulfills Eq. $(\ref{equation12})$ exactly when $\varsigma_0=0$, and
$$\varsigma(\lambda)=\int_0^\lambda (\lambda-\theta)^{\tilde{\iota}-1}\Psi_2(\lambda-\theta)\varrho(\theta,\varsigma_\theta+z_\theta)d\theta,\ \lambda\geq 0.$$
For any $\varsigma\in BC_{\tau,a,0}(\mathbb{E})$, we have
\begin{eqnarray*}
\|\varsigma_\lambda+z_\lambda\|_{[-\kappa,0]}&\leq& \displaystyle{\sup_{-\kappa\leq \theta\leq 0}}|a|^{\uparrow}(\lambda)|a|^{\uparrow}(-\lambda)\|\varsigma(\lambda+\theta)\|+\displaystyle{\sup_{-\kappa\leq \theta\leq 0}}\|z(\lambda+\theta)\|\\
&\leq& |a|^{\uparrow}(\lambda)\displaystyle{\sup_{0\leq \theta\leq \lambda}}|a|^{\uparrow}(-\theta)\|\varsigma(\theta)\|+M\|\eta(0)\|+\|\eta\|_{[-\kappa,0]}\\
&\leq& \psi(\lambda),
\end{eqnarray*}
where $\psi(\lambda):=|a|^{\uparrow}(\lambda)\|\varsigma\|_{b\tau a}+M\|\eta(0)\|+\|\eta\|_{[-\kappa,0]}$.

Thus, by means of $(J_5)$ we derive
\begin{equation}\label{dbjdbj}
\|\varrho(\lambda,\varsigma_\lambda+z_\lambda)\|\leq N(\lambda)\|\varsigma_\lambda+z_\lambda\|_{[-\kappa,0]}+|a|\|\varrho(\lambda,0)\|\leq \psi(\lambda)N(\lambda)+|a|\|\varrho(\lambda,0)\|.
\end{equation}

Let us define the operator $\mathcal{V}:\ BC_{\tau,a,0}(\mathbb{E})\rightarrow BC_{\tau,a,0}(\mathbb{E})$ via:
\[
(\mathcal{V}\varsigma)(\lambda)=
\left\{
\begin{array}{lr}
0,\ \qquad \qquad \qquad \qquad \qquad \qquad \qquad \ \ \ \ \lambda\in [-\kappa,0],\\
\displaystyle{\int_0^\lambda(\lambda-\theta)^{\tilde{\iota}-1}\Psi_2(\lambda-\theta)\varrho(\theta,\varsigma_\theta+z_\theta)d\theta,\ \lambda\geq 0.}
\end{array}
\right.
\]
Based on (\ref{dnkhfkhis3}), (\ref{nkfn}), (\ref{dbjdbj}) and $(J_2)$, we obtain
\begin{eqnarray}
&&\sup_{\lambda\geq 0}\left\|\displaystyle{a^{\uparrow}(-\lambda)\int_0^\lambda(\lambda-\theta)^{\tilde{\iota}-1}\Psi_2(\lambda-\theta)\varrho(\theta,\varsigma_\theta+z_\theta)d\theta}\right\|\nonumber\\
&\leq& \sup_{\lambda\geq 0}\frac{M}{\Gamma(\tilde{\iota})}\cdot |a|^{\uparrow}(-\lambda)\int_0^\lambda \psi(\theta)N(\theta)(\lambda-\theta)^{\tilde{\iota}-1}d\theta+\frac{M\Lambda|a|}{\Gamma(\tilde{\iota})}<\infty.
\end{eqnarray}
Thus, the well-definedness of the operator $\mathcal{V}$ is established. We observe that $z$ plus any fixed point of $\mathcal{V}$ is a mild solution to Eq. (\ref{equation1}). Next, we proceed to demonstrate that $\mathcal{V}$ maps into the space ${\rm{SAP}}_{\tau,a,0}([0,\infty),\mathbb{E})$.  To this end, we require the following claim.
\begin{Claim}\label{dbdkj}
Let $h(\lambda):\ [-\kappa,\infty)\rightarrow \mathbb{E}$ be a function with $h_0\in C([-\kappa,0],\mathbb{E})$ and $h|_{[0,\infty)}\in {\rm{SAP}}_{\tau,a}([0,\infty),\mathbb{E})$.
Then the function $ h_{\lambda}\in {\rm{SAP}}_{\tau,a}([0,\infty),C([-\kappa,0],\mathbb{E}))$.
\end{Claim}
For the sake of completeness, we include a brief proof of Claim \ref{dbdkj}: Obviously for any fixed $\lambda\geq 0$, $h_\lambda\in C([-\kappa,0],\mathbb{E})$, so there exists $\lambda_0\in [-\kappa,0]$ such that
$$\|h_{\lambda+\tau}-ah_\lambda\|_{[-\kappa,0]}=\sup_{-\kappa\leq \theta\leq 0}\|h(\lambda+\tau+\theta)-ah(\lambda+\theta)\|=\|h(\lambda+\tau+\lambda_0)-ah(\lambda+\lambda_0)\|.$$
Considering the substitution $\varkappa=\lambda+\lambda_0$, we then have
$$\lim_{\lambda\rightarrow \infty}\|a^\uparrow(-\lambda)[h(\lambda+\tau+\lambda_0)-ah(\lambda+\lambda_0)]\|=\lim_{\varkappa\rightarrow \infty}|a|^\uparrow(\lambda_0)\|a^\uparrow(-\varkappa)[h(\varkappa+\tau)-ah(\varkappa)]\|=0,$$
which is the desired result of Claim \ref{dbdkj}.

 Given any $\varsigma\in {\rm{SAP}}_{\tau,a,0}([0,\infty),\mathbb{E})$, it follows from (\ref{dnkfhfkhis2}) that the restriction
$z|_{[0,\infty)}$ belongs to ${\rm{SAP}}_{\tau,a}([0,\infty),\mathbb{E})$. Subsequently, by applying Claim \ref{dbdkj}, we deduce that the function $z_t$ lies in  ${\rm{SAP}}_{\tau,a}([0,\infty),C([-\kappa,0],\mathbb{E})).$ And we can calculate
\begin{eqnarray*}
&&\left\|a^{\uparrow}(-\lambda)[(\mathcal{V}\varsigma)(\lambda+\tau)-a(\mathcal{V}\varsigma)(\lambda)]\right\|\\
&=&\Bigg\|a^{\uparrow}(-\lambda)\Bigg[\int_0^\tau\frac{\Psi_2(\lambda+\tau-\theta)}{(\lambda+\tau-\theta)^{1-\tilde{\iota}}}\varrho(\theta,\varsigma_\theta+z_\theta)d\theta+\int_\tau^{\lambda+\tau}\frac{\Psi_2(\lambda+\tau-\theta)}{(\lambda+\tau-\theta)^{1-\tilde{\iota}}}\varrho(\theta,\varsigma_\theta+z_\theta)d\theta\\
&&-a\int_0^\lambda(\lambda-\theta)^{\tilde{\iota}-1}\Psi_2(\lambda-\theta)\varrho(\theta,\varsigma_\theta+z_\theta)d\theta\Bigg]\Bigg\|\\
&\leq&\frac{M}{\Gamma(\tilde{\iota})}|a|^{\uparrow}(-\lambda)\int_0^\tau(\lambda+\tau-\theta)^{\tilde{\iota}-1}\left[\psi(\theta)N(\theta)+|a|\|\varrho(\theta,0)\|\right]d\theta\\
&&\ +\int_0^\lambda|a|^{\uparrow}(-\lambda)(\lambda-\theta)^{\tilde{\iota}-1}\|\Psi_2(\lambda-\theta)\|\|\varrho(\theta+\tau,\varsigma_{\theta+\tau}+z_{\theta+\tau})-\varrho(\theta,\varsigma_{\theta+\tau}+z_{\theta+\tau})\|d\theta\\
&&\ +\frac{M|a|^{\uparrow}(-\lambda)}{\Gamma(\tilde{\iota})}\int_0^\lambda(\lambda-\theta)^{\tilde{\iota}-1}\|\varrho(\theta,\varsigma_{\theta+\tau}+z_{\theta+\tau})-a\varrho(\theta,\varsigma_\theta+z_\theta)\|d\theta\\
&=&\mathbb{J}_1(\lambda)+\mathbb{J}_2(\lambda)+\mathbb{J}_3(\lambda).
\end{eqnarray*}
By virtue of the inequality $\displaystyle{\lambda+\tau-\theta\geq \frac{\lambda+\tau}{\tau}(\tau-\theta)}$, we get
\begin{eqnarray*}
&&\int_0^\tau(\lambda+\tau-\theta)^{\tilde{\iota}-1}\left[\psi(\theta)N(\theta)+|a|\|\varrho(\theta,0)\|\right]d\theta\\
&\leq& \left(\frac{\tau}{\lambda+\tau}\right)^{1-\tilde{\iota}}\int_0^\tau(\tau-\theta)^{\tilde{\iota}-1}\left[\psi(\theta)N(\theta)+|a|\|\varrho(\theta,0)\|\right]d\theta,
\end{eqnarray*}
which, along with $|a|\geq 1$, leads to
\begin{center}
$\mathbb{J}_1(\lambda)\rightarrow 0$  as $\lambda\rightarrow \infty$.
\end{center}

From $(\ref{dbjdbj})$, we have
\begin{eqnarray*}
&&\|\varrho(\theta+\tau,\varsigma_{\theta+\tau}+z_{\theta+\tau})-\varrho(\theta,\varsigma_{\theta+\tau}+z_{\theta+\tau})\|\\
&\leq&\psi(\theta+\tau)[N(\theta+\tau)+N(\theta)]+|a|[\|\varrho(\theta+\tau,0)\|+\|\varrho(\theta,0)\|].
\end{eqnarray*}

According to $(J_3)$, given any $\mu>0$, we can find a constant $T^{\prime}>0$ satisfying
$$\|\varrho(\theta+\tau,\varsigma_{\theta+\tau}+z_{\theta+\tau})-\varrho(\theta,\varsigma_{\theta+\tau}+z_{\theta+\tau})\|<\mu,\ \theta\geq T^{\prime}.$$

In light of $(\ref{hejek})$ and $\displaystyle{\lambda-\theta\geq \frac{\lambda}{T^{\prime}}(T^{\prime}-\theta)}$, we obtain
\begin{eqnarray}
&&\mathbb{J}_2(\lambda)\nonumber\\
&\leq&\frac{M}{\Gamma(\tilde{\iota})}\int_0^{T^{\prime}}\frac{\psi(\theta+\tau)N(\theta+\tau)+|a|\|\varrho(\theta+\tau,0)\|+\psi(\theta+\tau)N(\theta)
+|a|\|\varrho(\theta,0)\|}{(\lambda-\theta)^{1-\tilde{\iota}}}d\theta\nonumber\\
&&+\mu\int_{T^{\prime}}^\lambda(\lambda-\theta)^{\tilde{\iota}-1}\|\mathbb{W}(\lambda-\theta)\|d\theta\nonumber\\
&\leq&\frac{M}{\Gamma(\tilde{\iota})}\Bigg[\int_\tau^{T^{\prime}+\tau}\frac{\psi(\theta)N(\theta)+|a|\|\varrho(\theta,0)\|}{(\lambda+\tau-\theta)^{1-\tilde{\iota}}}d\theta
+\int_0^{T^{\prime}}\frac{\psi(\theta+\tau)N(\theta)+|a|\|\varrho(\theta,0)\|}{(\lambda-\theta)^{1-\tilde{\iota}}}d\theta\Bigg]\nonumber\\
&&+\frac{M\mu}{\alpha}\nonumber\\
&\leq&\frac{M}{\Gamma(\tilde{\iota})}\left(\frac{T^{\prime}+\tau}{\lambda}\right)^{1-\tilde{\iota}}\int_0^{T^{\prime}+\tau}\frac{\psi(\theta)N(\theta)
+|a|\|\varrho(\theta,0)\|}{(T^{\prime}+\tau-\theta)^{1-\tilde{\iota}}}d\theta\nonumber\\
&&+\frac{M}{\Gamma(\tilde{\iota})}\left(\frac{T^{\prime}}{\lambda}\right)^{1-\tilde{\iota}}\int_0^{T^{\prime}}\frac{\psi(\theta+\tau)N(\theta)+|a|\|\varrho(\theta,0)\|}{(T^{\prime}-\theta)^{1-\tilde{\iota}}}d\theta
+\frac{M\mu}{\alpha}\nonumber\\
&\leq&\frac{2M}{\Gamma(\tilde{\iota})}\left(\frac{T^{\prime}+\tau}{\lambda}\right)^{1-\tilde{\iota}}|a|^\uparrow(T^{\prime}+\tau)[\psi(T^{\prime}+\tau)\Phi_2+|a|\Lambda]+\frac{M\mu}{\alpha}.\nonumber
\end{eqnarray}

Given that $\varsigma_\lambda+z_\lambda\in {\rm{SAP}}_{\tau,a}([0,\infty),C([-\kappa,0],\mathbb{E}))$, it follows that there exists a number $T^{\prime\prime}>0$ satisfying
$$|a|^\uparrow(-\theta)\|(\varsigma_{\theta+\tau}+z_{\theta+\tau})-a(\varsigma_\theta+z_\theta)\|_{[-\kappa,0]}\leq \mu,\ \theta\geq T^{\prime\prime},$$
which leads to the following estimation of $\mathbb{J}_3(\lambda)$:
\begin{eqnarray*}
\mathbb{J}_3(\lambda)&\leq&\frac{2M|a|}{\Gamma(\tilde{\iota})}\left[\int_0^{T^{\prime\prime}}\frac{\psi(\theta+\tau)N(\theta)}{(\lambda-\theta)^{1-\tilde{\iota}}}d\theta
+|a|\int_0^{T^{\prime\prime}}\frac{\|\varrho(\theta,0)\|}{(\lambda-\theta)^{1-\tilde{\iota}}}d\theta\right]\\
&&+\frac{M\mu |a|^{\uparrow}(-\lambda)}{\Gamma(\tilde{\iota})}\int_{T^{\prime\prime}}^\lambda\frac{|a|^\uparrow(\theta)N(\theta)}{(\lambda-\theta)^{1-\tilde{\iota}}}d\theta\\
&\leq&\frac{M}{\Gamma(\tilde{\iota})}\left[2|a||a|^\uparrow(T^{\prime\prime})(\psi(T^{\prime\prime}+\tau)\Phi_2
+|a|\Lambda)\left(\frac{T^{\prime\prime}}{\lambda}\right)^{1-\tilde{\iota}}+\mu\Phi_2\right].
\end{eqnarray*}
Based on the above analysis, we can establish that as $\lambda\rightarrow\infty$ and $\mu\rightarrow 0$,
\begin{center}
$\left\|a^{\uparrow}(-\lambda)[(\mathcal{V}\varsigma)(\lambda+\tau)-a(\mathcal{V}\varsigma)(\lambda)]\right\|\rightarrow 0$.
\end{center}
Hence, the operator $\mathcal{V}$ maps the space ${\rm{SAP}}_{\tau,a,0}([0,\infty),\mathbb{E})$ into itself.
Moreover, for any $\varsigma_i\in {\rm{SAP}}_{\tau,a,0}([0,\infty),\mathbb{E})\ (i=1,2)$, we have
\begin{eqnarray*}
\|a^\uparrow(-\lambda)[(\mathcal{V}\varsigma_1)(\lambda)-(\mathcal{V}\varsigma_1)(\lambda)]\|&\leq& \frac{M|a|^\uparrow(-\lambda)}{\Gamma(\tilde{\iota})}\int_0^\lambda(\lambda-\theta)^{\tilde{\iota}-1}L(\theta)\|\varsigma_{1\theta}-\varsigma_{2\theta}\|_{[-\kappa,0]}d\theta\\
&\leq&\frac{M|a|^\uparrow(-\lambda)}{\Gamma(\tilde{\iota})}\int_0^\lambda(\lambda-\theta)^{\tilde{\iota}-1}L(\theta)|a|^\uparrow(\theta)d\theta\|\varsigma_1-\varsigma_2\|_{b\tau a}.
\end{eqnarray*}
Therefore,
$$\|\mathcal{V}\varsigma_1-\mathcal{V}\varsigma_2\|_{b\tau a}\leq \frac{M\Phi_1}{\Gamma(\tilde{\iota})}\|\varsigma_1-\varsigma_2\|_{b\tau a},$$
which, per condition $(J_4)$, implies that $\mathcal{V}$ is a contraction mapping. The complete proof follows by invoking the
Banach contraction mapping principle.
\end{proof}

\section{Applications}
In this section, we give two examples to illustrate the main results obtained in previous sections.
\begin{Example}
Let $\mathbb{E}=\big\{\nu(\iota)|\ \nu(\iota)\in L^2[0,\pi],\ \nu(\iota)$ is a real function$\big\}$ and define its natural norm and inner product  for $\nu_i\in \mathbb{E}\ (i=1,2)$ by
$$\|\nu_1\|=\left(\int_{0}^{\pi}\nu_1^2(\iota)d\iota\right)^{\frac{1}{2}},\ \langle \nu_1,\nu_2\rangle=\int_{0}^{\pi}\nu_1(\iota)\nu_2(\iota)d\iota.$$

We analyze the time-fractional evolution  system below in the Banach space $\mathbb{E}$:
 \begin{equation}\label{dndjk}
\left\{
\begin{array}{lr}
^{c}\mathcal{D}^{\tilde{\iota}}_{\lambda} \nu(\lambda,\iota)=\Upsilon\nu(\lambda,\iota)+\varrho(\lambda,\nu(\lambda,\iota)),\ \lambda> 0,\ \iota\in [0,\pi]\\
\nu(0,\iota)=\nu_0(\iota),\ \iota\in [0,\pi],
\end{array}
\right.
\end{equation}
where $\tilde{\iota}\in (0,1)$, and $$\varrho(\lambda,\iota)=\varrho_1(\lambda,\iota)+\varrho_2(\lambda,\iota):=h_1(\lambda){\rm{arctan}}(-\|h_2(\lambda)\iota\|)+h_3(\lambda)h_4(\lambda)
{\rm{cos}}(\|h_5(\lambda)\iota\|),$$ $h_1,\ h_4\in \mathcal{P}_{\tau,a}([0,\infty),\mathbb{E})$, $h_2,\ h_5\in \mathcal{P}_{\tau,1/a}([0,\infty),\mathbb{R})$, $h_3\in L^\infty([0,\infty),\mathbb{R}) \cap L^1([0,\infty),\mathbb{R})$. $\xi(\lambda)=a^\uparrow(-\lambda)+b^\uparrow(-\lambda)$ with $|b|>|a|>1$. We have $\xi\in \widetilde{\mathbb{U}}$.

The operator $\Upsilon:\ D(\Upsilon)\subset \mathbb{E}\rightarrow \mathbb{E}$ is defined by
$$D(\Upsilon):=\left\{\nu\in \mathbb{E}:\ \nu^{\prime\prime}\in \mathbb{E},\ \nu(0)=\nu(\pi)=0\right\},\ \Upsilon\nu=\frac{\partial^2\nu}{\partial \iota^2}.$$

 It is known that the operator $-\Upsilon$ possesses a discrete spectrum whose eigenvalues take the form $k\ (k\in \mathbb{N}_{+})$, and the corresponding normalized eigenfunctions are expressed as $\varpi_k=\sqrt{\frac{2}{\pi}}{\rm{sin}}(k\iota).$ Furthermore, the operator $\Upsilon$ generates a compact analytic semigroup
$\{\Psi(\lambda)\}_{\lambda\geq 0}$ on $\mathbb{E}$, and $$\Psi(\lambda)\nu=\displaystyle{\sum_{k=1}^{\infty}}{\rm{e}}^{-k^2\lambda}\langle \nu,\varpi_k \rangle \varpi_k,\ \nu\in \mathbb{E}.$$

It is straightforward to verify  that $\|\Psi(\lambda)\|\leq {\rm{e}}^{-\lambda},\ \lambda\geq 0$. Thus, we choose $M=1.$
Take $w(\lambda)=1+\lambda^2$. Clearly, $w_r\sim w$ for any $r\in \mathbb{R}$.

 Now, recalling  $h_1,\ h_4\in \mathcal{P}_{\tau,a}([0,\infty),\mathbb{E})$, $h_2,\ h_5\in \mathcal{P}_{\tau,1/a}([0,\infty),\mathbb{R})$, $h_3\in L^\infty([0,\infty),\mathbb{R})\cap L^1([0,\infty),\mathbb{R})$, we can obtain
\begin{eqnarray*}
&&\sup_{\lambda\geq 0}\left\|\xi(\lambda)\varrho(\lambda,\iota)\right\|\\
&=&\sup_{\lambda\geq 0}\left\|\xi(\lambda)\left[h_1(\lambda){\rm{arctan}}(-\|h_2(\lambda)\iota\|)+h_3(\lambda)h_4(\lambda){\rm{cos}}(\|h_5(\lambda)\iota\|)\right]\right\|\\
&\leq&\frac{\pi}{2}\sup_{\lambda\geq 0}\left\|\xi(\lambda)h_1(\lambda)\right\|+\|h_3\|_{L^\infty([0,\infty),\mathbb{R})}\sup_{\lambda\geq 0}\left\|\xi(\lambda)h_4(\lambda)\right\|<\infty,
\end{eqnarray*}
$$\varrho_1(\lambda+\tau,a\iota)=h_1(\lambda+\tau){\rm{arctan}}(-\|h_2(\lambda+\tau)a\iota\|)=ah_1(\lambda){\rm{arctan}}(-\|h_2(\lambda)\iota\|)=a\varrho_1(\lambda,\iota),$$
and
\begin{eqnarray*}
&&\lim_{\ell\rightarrow \infty}\frac{1}{\ell}\int_0^\ell \left\|\xi(\lambda)\varrho_2(\lambda,\iota)\right\|d\lambda=\lim_{\ell\rightarrow \infty}\frac{1}{\ell}\int_0^\ell \left\|\xi(\lambda)h_3(\lambda)h_4(\lambda){\rm{cos}}(\|h_5(\lambda)\iota\|)\right\|d\lambda=0,
\end{eqnarray*}
independently of $\iota\in \mathbb{E}$. Hence, $\varrho$ meets condition $(I_1)$. Moreover, it can be easily verified that
$$h_1(\lambda)h_2(\lambda),\ h_4(\lambda)h_5(\lambda)\in \mathcal{P}_\tau([0,\infty),\mathbb{E})\subseteq BC([0,\infty),\mathbb{E}).$$
\end{Example}
And for any $\iota_i\in \mathbb{E}\ (i=1,2)$, we have
\begin{eqnarray*}
&&\|\varrho(\lambda,\iota_1)-\varrho(\lambda,\iota_2)\|\\
&=&\|h_1(\lambda){\rm{arctan}}(-\|h_2(\lambda)\iota_1\|)+h_3(\lambda)h_4(\lambda){\rm{cos}}(\|h_5(\lambda)\iota_1\|)-h_1(\lambda){\rm{arctan}}(-\|h_2(\lambda)\iota_2\|)\\
&&-h_3(\lambda)h_4(\lambda){\rm{cos}}(\|h_5(\lambda)\iota_2\|)\|\\
&\leq&\|h_1(\lambda)[{\rm{arctan}}(-\|h_2(\lambda)\iota_1\|)-{\rm{arctan}}(-\|h_2(\lambda)\iota_2\|)]+\|h_3(\lambda)h_4(\iota)[{\rm{cos}}(\|h_5(\lambda)\iota_1\|)\\
&&-{\rm{cos}}(\|h_5(\lambda)\iota_2\|)]\|\\
&\leq&\|h_1(\lambda)h_2(\lambda)\|\|\iota_1-\iota_2\|+\|h_3(\lambda)h_4(\lambda)h_5(\lambda)\|\|\iota_1-\iota_2\|\\
&\leq&\left(\|h_1h_2\|_{L^\infty([0,\infty),\mathbb{E})}+\|h_3\|_{L^\infty([0,\infty),\mathbb{R})}\cdot\|h_4h_5\|_{L^\infty([0,\infty),\mathbb{E})}\right)\|\iota_1-\iota_2\|.
\end{eqnarray*}
Therefore, by Theorem \ref{hdjkhd} and Remark \ref{hddjlak}, there exists a unique mild solution $\nu\in {\rm{PSAP}}_{\xi,\tau,a}([0,\infty),\mathbb{E},w)$ for Syst. $(\ref{dndjk})$, provided that
$$\left(\|h_1h_2\|_{L^\infty([0,\infty),\mathbb{E})}+\|h_3\|_{L^\infty([0,\infty),\mathbb{R})}\cdot\|h_4h_5\|_{L^\infty([0,\infty),\mathbb{E})}\right)\cdot \left\|\xi(\cdot){\cdot}^{\tilde{\iota}-1}\Psi_2(\cdot)\right\|_{L^1((0,\infty),[0,\infty))}<1.$$
Fig. \ref{fig:6} depicts the solution $\nu(\cdot,\cdot)$ of Syst. $(\ref{dndjk})$ when $\tilde{\iota}=0.4$, $h_1(\lambda)=h_4(\lambda)={\rm{sin}}\left(\frac{2\pi \lambda}{\tau}\right)a^{\frac{\lambda}{\tau}}$, $h_2(\lambda)=h_5(\lambda)={\rm{cos}}\left(\frac{2\pi \lambda}{\tau}\right)a^{-\frac{\lambda}{\tau}}$ with $\tau=1,\ a=2$, $h_3(\lambda)=\frac{1}{\lambda^2+1}$, and $\nu_0\equiv0$.

\begin{Example}
Let $\mathbb{E}=\mathbb{R}^2$. For any vector $v=(v_1,v_2)^T\in \mathbb{E}$, define its norm as
$$\|v\|=|v_1|+|v_2|.$$
For a given matrix $\mathfrak{M}=(m_{\imath\jmath})_{2\times 2}$, its norm is defined as
$$\|\mathfrak{M}\|=\max_{1\leq \jmath\leq 2}\sum_{\imath=1}^2 |m_{\imath\jmath}|.$$
We study the finite delay time-fractional evolution system in the space $\mathbb{E}$, given by:
\begin{equation}\label{equation4}
\left\{
\begin{array}{lr}
^c\mathcal{D}^{\frac{1}{2}}_\lambda \nu_1(\lambda)=-2\nu_1(\lambda)+\frac{{\rm{sin}} 2\pi \lambda}{11(\lambda+\lambda^{\frac{7}{3}})^{\frac{1}{4}}}[2\nu_1(\lambda+\theta)+3\nu_2(\lambda+\theta)],\ \lambda>0,\\
^c\mathcal{D}^{\frac{1}{2}}_\lambda \nu_2(\lambda)=-2\nu_2(\lambda)+\frac{{\rm{cos}} 2\pi \lambda}{11(\lambda+\lambda^{\frac{7}{3}})^{\frac{1}{4}}}[\nu_1(\lambda+\theta)+4\nu_2(\lambda+\theta)],\ \lambda>0,\\
\nu_1(\lambda)=\nu_2(\lambda)=0.1,\ -1\leq \lambda\leq 0.
\end{array}
\right.
\end{equation}
Expressed in vector form, Syst. $(\ref{equation4})$ becomes
\[
\left\{
\begin{array}{lr}
^c\mathcal{D}^{\frac{1}{2}}_\lambda \nu(\lambda)=\Upsilon\nu(\lambda)+\widetilde{\mathcal{R}}(\lambda,\nu_\lambda),\ \lambda> 0,\\
\nu(\lambda)=0.1,\ \lambda\in[-1,0],
\end{array}
\right.
\]
where
$$\nu(\lambda)=\left(\nu_1(\lambda),\nu_2(\lambda)\right)^T,$$
$$\Upsilon=\left( {\begin{array}{cc}
-2&0\\
0&-2\\
\end{array}} \right),$$
$$\widetilde{\mathcal{R}}(\lambda,\nu_\lambda)=B(\lambda)\varrho(\nu_\lambda),$$
$$B(\lambda)=
\left( {\begin{array}{cc}
\frac{2{\rm{sin}} 2\pi \lambda}{11(\lambda+\lambda^{\frac{7}{3}})^{\frac{1}{4}}}&\frac{3{\rm{sin}} 2\pi \lambda}{11(\lambda+\lambda^{\frac{7}{3}})^{\frac{1}{4}}}\\
\frac{{\rm{cos}} 2\pi \lambda}{11(\lambda+\lambda^{\frac{7}{3}})^{\frac{1}{4}}}&\frac{4{\rm{cos}} 2\pi \lambda}{11(\lambda+\lambda^{\frac{7}{3}})^{\frac{1}{4}}}\\
\end{array}} \right),
$$
and
$$\varrho(\nu_\lambda)=(\nu_{1\lambda},\nu_{2\lambda})^T.$$
As established in \cite{Pazy}, the operator $\Upsilon$ generates a bounded semigroup, which can be written as
$$\Psi(\lambda)=e^{\Upsilon\lambda}=\left( {\begin{array}{cc}
{\rm{e}}^{-2\lambda}&0\\
0&{\rm{e}}^{-2\lambda}\\
\end{array}} \right).$$
Moreover, this semigroup satisfies the exponential bound
$$\|\Psi(\lambda)\|\leq {\rm{e}}^{-2\lambda},\ \lambda\geq 0,$$
where the bound parameters are given by $M=1$ and $\alpha=2$.

For $\vartheta_i\in C([-1,0],\mathbb{E})\ (i=1,2)$, one can easily observe that
$$\|\widetilde{\mathcal{R}}(\lambda,\vartheta_1)-\widetilde{\mathcal{R}}(\lambda,\vartheta_2)\|\leq\frac{5}{11(\lambda+\lambda^{\frac{7}{3}})^{\frac{1}{4}}}\|\vartheta_1-\vartheta_2\|_{[-1,0]}:=L(\lambda)\|\vartheta_1-\vartheta_2\|_{[-1,0]},$$
and
$$\|\widetilde{\mathcal{R}}(\lambda,\vartheta_1)-a\widetilde{\mathcal{R}}(\lambda,\vartheta_2)\|\leq\frac{5}{11(\lambda+\lambda^{\frac{7}{3}})^{\frac{1}{4}}}\|\vartheta_1-a\vartheta_2\|_{[-1,0]}:=N(\lambda)\|\vartheta_1-a\vartheta_2\|_{[-1,0]}.$$
From
$$\left(\lambda+\lambda^{\frac{7}{3}}\right)^{-\frac{1}{4}}\leq \lambda^{-\frac{1}{4}},$$
$$\int_0^\lambda (\lambda-\theta)^{-\frac{2}{3}}\theta^{-\frac{1}{3}}ds=B\left(\frac{1}{3},\frac{2}{3}\right)=\Gamma\left(\frac{1}{3}\right)\Gamma\left(\frac{2}{3}\right),$$
and
$$\int_0^\infty\frac{d\lambda}{1+\lambda^\upsilon}=\frac{\pi}{\upsilon{\rm{sin}}(\pi/\upsilon)}\ \left(\upsilon\in(1,2)\right),$$
one can conclude $\left(\lambda+\lambda^{\frac{7}{3}}\right)^{-\frac{1}{4}}\in L^1(0,\lambda)$ and
\begin{eqnarray*}
&&|a|^{\uparrow}(-\lambda)\int_0^\lambda |a|^{\uparrow}(\theta)(\lambda-\theta)^{-\frac{1}{2}}L(\theta)d\theta\\
&=&|a|^{\uparrow}(-\lambda)\int_0^\lambda |a|^{\uparrow}(\theta)(\lambda-\theta)^{-\frac{1}{2}}N(\theta)d\theta\\
&\leq&\frac{5}{11}\int_0^\lambda (\lambda-\theta)^{-\frac{1}{2}}\theta^{-\frac{1}{4}}\left(1+\theta^{\frac{4}{3}}\right)^{-\frac{1}{4}}d\theta\\
&\leq&\frac{5}{11}\left(\int_0^\lambda \left[(\lambda-\theta)^{-\frac{1}{2}}\theta^{-\frac{1}{4}}\right]^{\frac{4}{3}}d\theta\right)^{\frac{3}{4}}\cdot \left(\int_0^\lambda \left(\left(1+\theta^{\frac{4}{3}}\right)^{-\frac{1}{4}}\right)^4d\theta\right)^{\frac{1}{4}}\\
&\leq&\frac{5}{11}\left(\Gamma\left(\frac{1}{3}\right)\Gamma\left(\frac{2}{3}\right)\right)^{\frac{3}{4}}\left(\frac{3\sqrt{2}\pi}{4}\right)^{\frac{1}{4}}.
\end{eqnarray*}
So,
$$\frac{1}{\Gamma\left(\frac{1}{2}\right)}\sup_{\lambda\geq 0}|a|^{\uparrow}(-\lambda)\int_0^t |a|^{\uparrow}(\theta) (\lambda-\theta)^{-\frac{1}{2}}L(\theta)d\theta\leq \frac{5\left(\Gamma\left(\frac{1}{3}\right)\Gamma\left(\frac{2}{3}\right)\right)^{\frac{3}{4}}}{11\Gamma\left(\frac{1}{2}\right)}\left(\frac{3\sqrt{2}\pi}{4}\right)^{\frac{1}{4}}\approx 0.91<1.$$
Clearly, $(J_3)$ is satisfied. Also, we have
\begin{eqnarray*}
\|\widetilde{\mathcal{R}}(\lambda+1,\vartheta)-\widetilde{\mathcal{R}}(\lambda,\vartheta)\|&\leq& \frac{5}{11}\left[\frac{1}{((\lambda+1)+(\lambda+1)^{\frac{7}{3}})^{\frac{1}{4}}}+\frac{1}{(\lambda+\lambda^{\frac{7}{3}})^{\frac{1}{4}}}\right]\|\vartheta\|_{[-1,0]}\\
&\rightarrow&0,\ t\rightarrow\infty.
\end{eqnarray*}
Consequently, the conditions in Theorem \ref{cnknn} are satisfied. Therefore, according to Theorem \ref{cnknn}, Syst. $(\ref{equation4})$ admits a unique $\emph{S}$-asymptotically $(1,a)$-periodic mild solution for each $|a|\geq1$. Fig. \ref{fig:7} demonstrates the curved surface of $(\cdot,\nu_1(\cdot),\nu_2(\cdot))$ of Syst. $(\ref{equation4})$ on $[0,\infty)$.
\end{Example}

\begin{figure}[h!]
\centering
  {
    \raisebox{0.08\height}{
     \begin{minipage}{7cm}
            \includegraphics[height=4.75cm,width=6.5cm]{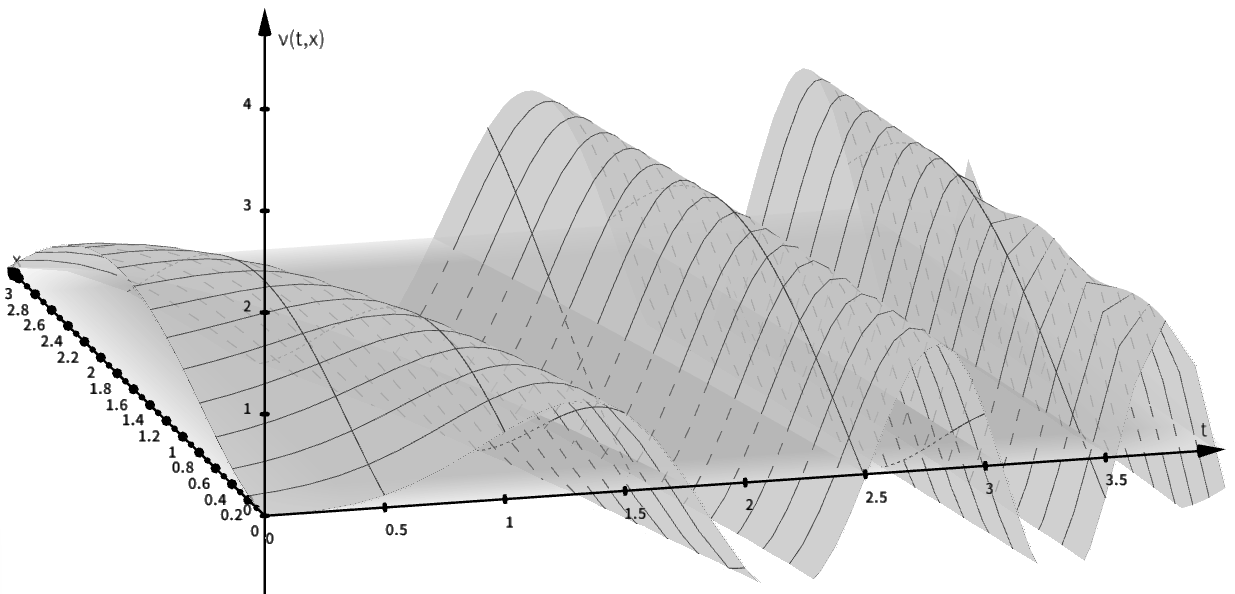}
            \caption{Solution $\nu(\cdot,\cdot)$ of Eq. $(\ref{dndjk})$}
            \label{fig:6}
       \end{minipage}}
      \begin{minipage}{7cm}
            \includegraphics[height=4.1cm,width=6.5cm]{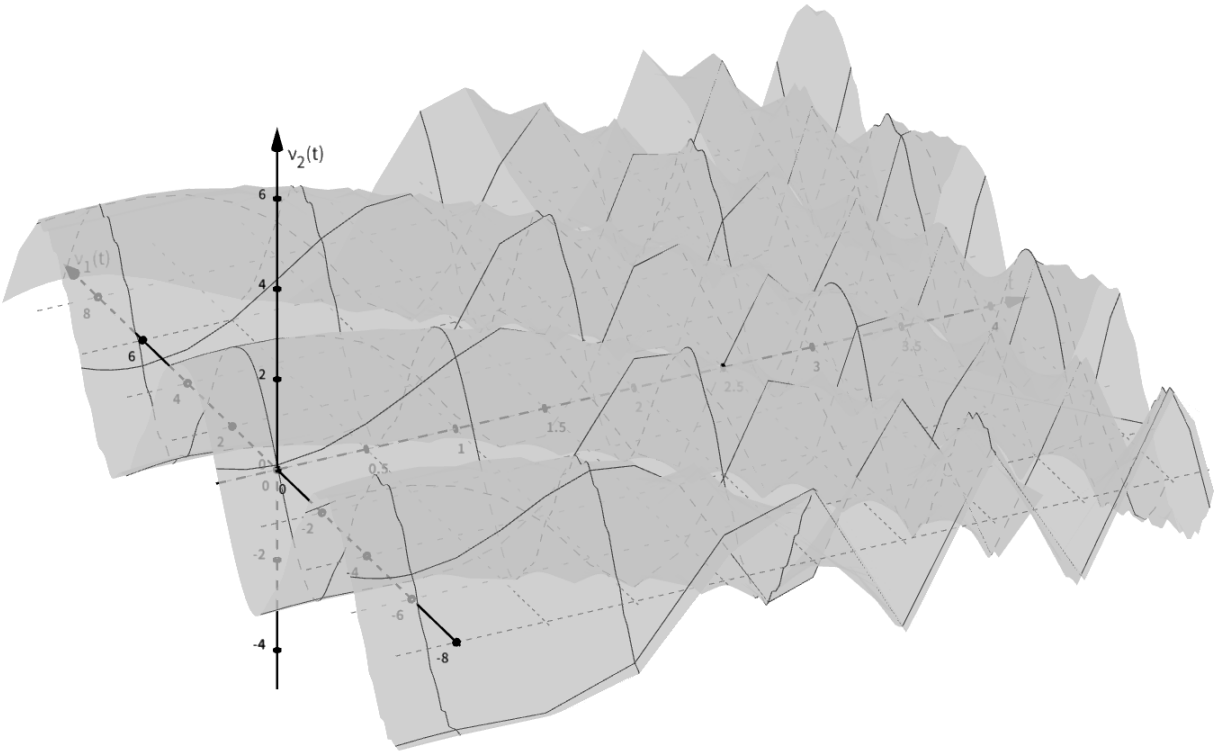}
            \caption{Curved surface of $(\cdot,\nu_1(\cdot),\nu_2(\cdot))$ of Eq. $(\ref{equation4})$ on $[0,\infty)$}
            \label{fig:7}
       \end{minipage}
   }
\end{figure}

\vspace{0.4cm}

{\bf  Authors' contribution statement}

Jin Liang: Writing - review \& editing, Supervision, Project administration, Methodology, 
Funding acquisition, Formal analysis, Conceptualization.  Yunyi Mu: Writing - original draft, Writing - review \& editing, Conceptualization, Methodology, Investigation, Formal analysis.
Gaston Mandata N'Gu\'er\'ekata: Writing - review \& editing, Formal analysis, Conceptualization.
Ti-Jun Xiao: Writing - review \& editing, Supervision, Project administration, Methodology, Formal analysis, Conceptualization.

\vspace{0.4cm}

{\bf Funding Declaration Statement} 

The authors declare the following financial interests/personal relationships which may be considered as potential competing
interests: Jin Liang reports financial support was provided by National Natural Science Foundation of China (Grant No. 12371116). Ti-Jun Xiao reports financial supports were provided by National Natural Science Foundation of China (Grant No. 12171094) and the Shanghai Key Laboratory for Contemporary Applied Mathematics (Grant No. 08DZ2271900). If there are other authors, they declare that they have no known competing financial interests or personal relationships that could have appeared to influence the work reported in this paper.

\vspace{0.4cm}

{\bf Conflict of interest}

The authors have no competing interests to declare.

\vspace{0.4cm}

{\bf Availability of Data and Materials}

No data and materials were generated or used during the current study.

\end{document}